\documentclass[reqno,12pt]{amsart}
\usepackage{amssymb}
\usepackage{mathrsfs}
\usepackage{amsmath,amssymb}
\usepackage{paralist}
\usepackage{graphics} 
\usepackage{epsfig} 
\usepackage[htt]{hyphenat}
\usepackage[colorlinks = true]{hyperref}
\hypersetup{urlcolor = red, citecolor = blue}
\usepackage[top = 1in,bottom = 1in,left = 1in,right = 1in]{geometry}
\usepackage{cite}
\usepackage{lineno}
\newtheorem{thm}{Theorem}[section]

\newtheorem{lem}[thm]{Lemma}

\theoremstyle{definition}
\newtheorem{defn}{Definition}[section]

\numberwithin{equation}{section}

\def\d{\,\mathrm{d}}
\def\dt{\frac{\mathrm{d}}{\mathrm{d}t}}

\allowdisplaybreaks[3]

\begin{document}

\title[ZHANG AND LI] 
{Boundedness in a two-dimensional doubly degenerate nutrient taxis system with logistic source}
\author[Zhang]{Zhiguang Zhang }%
\address{School of Mathematics, Southeast University, Nanjing 211189, P. R. China; and School of Mathematics and Statistics, Chongqing Three Gorges University, Wanzhou 404020 , P. R. China}
\email{guangz$\_$z@163.com}

\author[Li]{Yuxiang Li}
\address{School of Mathematics, Southeast University, Nanjing 211189, P. R. China}
\email{lieyx@seu.edu.cn}

\thanks{Supported in part by National Natural Science Foundation of China (No. 12271092, No. 11671079), Jiangsu Provincial Scientific Research Center of Applied Mathematics (No. BK20233002) and Science and Technology Research Program of Chongqing Municipal Education Commission (No. KJQN202201226).}



\begin{abstract}
We are concerned with the following doubly degenerate nutrient taxis system
\begin{align}
\begin{cases}\tag{$\star$}\label{eq-0.1}
u_t=\nabla\cdot(u v\nabla u)-\nabla\cdot(u^{2} v\nabla v)+u-u^2,\\[1mm]
v_t=\Delta v-u v,
\end{cases}
\end{align}
posed in a bounded smooth domain $\Omega\subset\mathbb{R}^2$ under homogeneous Neumann boundary conditions.  
This model was introduced to describe the aggregation patterns of colonies of \emph{Bacillus subtilis} observed on thin agar plates. 
Previous results have established global boundedness in one space dimension and, in two dimensions, under additional assumptions such as small initial data or convex domains (see, e.g., M. Winkler, \textit{Trans. Amer. Math. Soc.}, 2021; M. Winkler, \textit{J. Differ. Equ.}, 2024).
In the presence of the quadratic degradation term in the logistic growth, which markedly enhances the dissipative structure of the system, and by employing a weighted energy method, we prove that for arbitrary smooth initial data the problem \eqref{eq-0.1} admits a global weak solution that remains uniformly bounded in time.  
\end{abstract}
\maketitle

\section{Introduction}\label{section1}

Intricate patterns can emerge in bacterial colonies under various environmental conditions.  
For instance, changes in nutrient availability or the introduction of an attractant can induce different forms of aggregation.  Such patterns can be reproduced in vitro on agar plates. 
In this context, for certain bacterial species such as \emph{Bacillus subtilis}, several studies \cite{1993-CRiM-MatsuyamaMatsushita,1994-N-Ben-JacobSchochetTenenbaumCohenCzirokVicsek} have examined the geometry of these aggregations. In \cite{1992-PASMaiA-Fujikawa,1989-JotPSoJ-FujikawaMatsushita,1990-PASMaiA-MatsushitaFujikawa,1992-JotPSoJ-OhgiwariMatsushitaMatsuyama}, the shapes of different aggregations are analyzed with respect to varying agar and nutrient concentrations. For rigid media—those with high agar concentrations—in the presence of low nutrient levels, complex branching formations have been observed. From a mathematical perspective, nutrient-taxis systems of the following form were considered in \cite{1997-JoTB-KawasakiMochizukiMatsushitaUmedaShigesada} to model these phenomena:
\begin{align}\label{sys:Kaw}
\begin{cases}
u_t=\nabla \cdot (D_u(u,v) \nabla u) + u v, \\[1mm]
v_t=D_v \Delta v - u v,
\end{cases}
\end{align}
where $D_u(u,v)$ denotes the diffusion coefficient of the bacterial cells, and $D_v$ is the constant diffusion coefficient of the nutrient. Experimental evidence suggests that bacteria are essentially immotile when either $u$ or $v$ is low, whereas their motility increases as $u$ or $v$ increases. Motivated by this observation, the authors proposed the simplest diffusion coefficient
\[
D_u(u,v)=u v.
\]
The system \eqref{sys:Kaw}, equipped with homogeneous Neumann boundary conditions in a smoothly bounded convex domain $\Omega \subset \mathbb{R}^n$, was subsequently studied by Winkler \cite{2022-CVPDE-Winkler}, where stabilization of arbitrary structures was established under suitable assumptions on the initial data.

To more accurately describe the formation of such aggregation patterns, Leyva et al.~\cite{2013-PA-LeyvaMalagaPlaza} extended the degenerate diffusion model \eqref{sys:Kaw} to the following doubly degenerate nutrient-taxis system:
\begin{align}\label{SYS:LEYVA}
\begin{cases}
u_t=\nabla \cdot(u v \nabla u)- \nabla \cdot\left(u^{2} v \nabla v\right)+ u v, \\[2mm]
v_t=\Delta v - u v,
\end{cases}
\end{align}
Numerical simulations in \cite{1997-JoTB-KawasakiMochizukiMatsushitaUmedaShigesada, 2000-AiP-Ben-JacobCohenLevine, 2013-PA-LeyvaMalagaPlaza} indicate that, depending on the initial data and parameter conditions, the model \eqref{SYS:LEYVA} can generate a variety of rich branching patterns that are very close to those observed in biological experiments. In \cite{2021-TAMS-Winkler}, Winkler studied \eqref{SYS:LEYVA} in one space dimension, namely the cross-diffusion system
\begin{align}\label{ONESYS:WINKLER}
\begin{cases}
u_t=(u v u_x)_x-(u^2 v v_x)_x+u v, & x\in\Omega,\ t>0,\\[1mm]
v_t=v_{xx}-u v, & x\in\Omega,\ t>0,\\[1mm]
u v u_x-u^2 v v_x=0,\quad v_x=0, & x\in\partial\Omega,\ t>0,\\[1mm]
u(x,0)=u_0(x),\quad v(x,0)=v_0(x), & x\in\Omega,
\end{cases}
\end{align}
with initial data assumed to satisfy
\begin{align}\label{initial-1}
\begin{cases}
u_0\in C^{\vartheta}(\overline{\Omega})\ \text{for some }\vartheta\in(0,1),\quad u_0\geqslant 0,\quad \displaystyle\int_{\Omega}\ln u_0>-\infty,\\[1mm]
v_0\in W^{1,\infty}(\Omega),\quad v_0>0\ \text{in }\overline{\Omega}.
\end{cases}
\end{align}
Using energy estimates, Winkler proved that \eqref{ONESYS:WINKLER} admits a global weak solution which is uniformly bounded in time and converges to an equilibrium in an appropriate topology. 
Subsequently, Li and Winkler \cite{2022-CPAA-LiWinkler} removed the integrability assumption $\int_{\Omega}\ln u_0>-\infty$ in \eqref{initial-1} and obtained analogous results.
In the two-dimensional setting, consider the variant of \eqref{SYS:LEYVA}
\begin{align}\label{0930-1047}
\begin{cases}
u_t=\nabla\cdot(u v \nabla u)-\nabla\cdot\bigl(u^\alpha v \nabla v\bigr)+u v,\\[1mm]
v_t=\Delta v-u v,
\end{cases}
\end{align}
in a smoothly bounded, convex domain $\Omega\subset\mathbb{R}^2$ with homogeneous Neumann boundary conditions, where $\alpha>0$. 
It was shown in \cite{2022-JDE-Li} that the system admits a global weak solution for certain parameter regimes, in particular when $\alpha\in(1,\tfrac{3}{2})$ in two space dimensions, and when $\alpha\in(\tfrac{7}{6},\tfrac{13}{9})$ in three space dimensions. 
In \cite{2024-JDE-Winkler}, Winkler treated a more general class of variants of \eqref{SYS:LEYVA} (including \eqref{0930-1047} as a special case) and obtained global, uniformly bounded weak solutions in bounded convex planar domains under either of the following alternatives: $\alpha<2$ with reasonably regular (possibly large) initial data, or $\alpha=2$ provided the initial nutrient profile $v_0$ satisfies a suitable smallness condition. Recently, Pan \cite{2024-NARWA-Pan} proposed a related variant of \eqref{SYS:LEYVA},
\[
\begin{cases}
u_t=\nabla\cdot(u v \nabla u)-\nabla\cdot(u^2 v \nabla v)+\rho u-\mu u^\kappa,\\[1mm]
v_t=\Delta v-u v,
\end{cases}
\]
and established global existence of weak solutions in higher space dimensions provided \(\kappa>\tfrac{n+2}{2}\), \(\rho,\mu>0\), and the initial data are suitably regular (and may be arbitrarily large). 
In the borderline case \(\kappa=2\), Li and Winkler \cite{2024-AaA-LiWinkler} proved that the model admits global, continuous weak solutions for any reasonably regular initial data in two space dimensions. The above literature is primarily concerned with establishing the global existence of weak solutions to model~\eqref{SYS:LEYVA} in two space dimensions.

Motivated by the aforementioned works, in this paper we focus on the following doubly degenerate nutrient-taxis system with logistic source:
\begin{equation}\label{SYS:MAIN}
\begin{cases}
u_t=\nabla \cdot(u v \nabla u)- \nabla \cdot\bigl(u^{2}v \nabla v\bigr)+ u - u^2, & x \in \Omega, \ t>0,\\[1mm]
v_t=\Delta v-u v, & x \in \Omega, \ t>0,\\[1mm]
\bigl(u v \nabla u- u^{2} v \nabla v\bigr) \cdot \nu = \nabla v \cdot \nu = 0, & x \in \partial \Omega, \ t>0,\\[1mm]
u(x, 0)=u_0(x), \quad v(x, 0)=v_0(x), & x \in \Omega,
\end{cases}
\end{equation}
where \(\Omega \subset \mathbb{R}^2\) is a general smooth bounded domain. 
The scalar functions $u$ and $v$ denote the bacterial density and the nutrient concentration, respectively. 
We assume throughout that the initial data satisfy
\begin{equation}\label{assIniVal}
\begin{cases}
u_0 \in W^{1, \infty}(\Omega), \ u_0 \geqslant 0, \ u_0 \not\equiv 0,\\
v_0 \in W^{1, \infty}(\Omega), \ v_0 >0 \ \text{in } \overline{\Omega}.
\end{cases}
\end{equation}
The purpose of the present paper is to establish the global-in-time existence and uniform-in-time boundedness of solutions to \eqref{SYS:MAIN} with doubly degenerate diffusion on a general smooth bounded domain.

\vskip 3mm
Before stating the main result, we introduce the definition of the weak solution to system \eqref{SYS:MAIN}.
\begin{defn} \label{def-weak-sol}
Let $\Omega \subset \mathbb{R}^2$ be a bounded domain with smooth boundary. Suppose that $u_0 \in L^1(\Omega)$ and $v_0 \in L^1(\Omega)$ are nonnegative. By a global weak solution of the system (\ref{SYS:MAIN}) we mean a pair $(u, v)$ of functions satisfying
\begin{equation}\label{-2.1}
\begin{cases}
u \in L_{\mathrm{loc}}^1(\overline{\Omega} \times[0, \infty)) \quad \text { and } \\
v \in L_{\mathrm{loc}}^{\infty}(\overline{\Omega} \times[0, \infty)) \cap L_{\mathrm{loc}}^1\left([0, \infty) ; W^{1,1}(\Omega)\right)
\end{cases}
\end{equation}
and
\begin{equation}\label{-2.2}
u^2 \in L_{\mathrm{loc}}^1\left([0, \infty) ; W^{1,1}(\Omega)\right) \quad \text { and } \quad u^{2} \nabla v \in L_{\mathrm{loc}}^1\left(\overline{\Omega} \times[0, \infty) ; \mathbb{R}^2\right),
\end{equation}
which are such that
\begin{align}\label{-2.3}
-\int_0^{\infty} \int_{\Omega} u \varphi_t-\int_{\Omega} u_0 \varphi(0)=&-\frac{1}{2} \int_0^{\infty} \int_{\Omega} v  \nabla u^{2} \cdot \nabla \varphi
+ \int_0^{\infty} \int_{\Omega} u^{2} v \nabla v \cdot \nabla \varphi \nonumber\\
& + \int_0^{\infty} \int_{\Omega} u  \varphi - \int_0^{\infty} \int_{\Omega} u^2  \varphi
\end{align}
and
\begin{align}\label{-2.4}
\int_0^{\infty} \int_{\Omega} v \varphi_t+\int_{\Omega} v_0 \varphi(0)=\int_0^{\infty} \int_{\Omega} \nabla v \cdot \nabla \varphi+\int_0^{\infty} \int_{\Omega} u v \varphi
\end{align}
for all $\varphi \in C_0^{\infty}(\overline{\Omega} \times[0, \infty))$.
\end{defn}

We are now in a position to state the main result of this paper.

\begin{thm} \label{thm-1.1}
Let $\Omega \subset \mathbb{R}^2$ be a bounded domain with smooth boundary. Assume that the initial data $\left(u_0, v_0\right)$ satisfies \eqref{assIniVal}. Then there exist functions
\begin{align}\label{solu:property}
\begin{cases}
u \in C^{0}(\overline{\Omega} \times[0, \infty)) \quad \text { and } \\
v \in C^0(\overline{\Omega} \times[0, \infty)) \cap C^{2,1}(\overline{\Omega} \times(0, \infty))
\end{cases}
\end{align}
such that $u \geqslant 0$ and $v>0$ in $\bar{\Omega} \times[0, \infty)$, and that $(u, v)$ solves \eqref{SYS:MAIN} in the sense of Definition \ref{def-weak-sol}, and that $(u, v)$ is bounded in the sense that
\begin{align}\label{solv-1.13}
\|u(t)\|_{L^\infty(\Omega)}+\|v(t)\|_{W^{1, \infty}(\Omega)} \leqslant C, \quad \text { for all } t>0
\end{align} 
where $C>0$ is a constant independent of t.
\end{thm}

\vskip 1mm


\section{Preliminaries}


Before going further, we list some lemmas, which will be used throughout this paper.

\begin{lem}[{\cite[Lemma 3.4]{2014-SJMA-StinnerSurulescuWinkler}}]\label{lem-0926-1}
Let $T>0$, $\tau \in(0, T)$, $a>0$ and $b>0$, and suppose that $z:[0, T) \rightarrow [0, \infty)$ is absolutely continuous such that
$$
z^{\prime}(t)+a z(t) \leqslant h(t)\quad \text { for a.e. } t \in(0, T)
$$
with some nonnegative function $h \in L_{\text {loc }}^1([0, T))$ satisfying
$$
\sup_{0 \leqslant t \leqslant T-\tau} \int_t^{t+\tau} h(s) d s \leqslant b. 
$$
Then
$$
z(t) \leqslant \max \left\{z(0)+b,~\frac{b}{a\tau}+2 b\right\} \quad \text { for all } t \in(0, T).
$$
\end{lem}

\begin{lem}[{\cite[Chapter~3, Lemma~1.1]{1997--Temam}}]\label{lem-0926-2}
Let $T>0$, $\tau \in\left(0, T\right)$. Suppose that $a$, $b$, $z$ are three positive locally integrable functions on $(0, T)$  such that $z^{\prime}(t)$ is locally integrable on $(0, T)$ and the following inequalities are satisfied:
$$
z^{\prime}(t) \leqslant a(t) z(t)+b(t)\quad \text { for all } t \in\left(0, T\right)
$$
as well as
$$
\sup_{0 \leqslant t \leqslant T-\tau} \int_t^{t+\tau} a(s)ds \leqslant a_1,~~ \sup_{0 \leqslant t \leqslant T-\tau} \int_t^{t+\tau} b(s)ds   \leqslant a_2 \text { and } ~~ \sup_{0 \leqslant t \leqslant T-\tau} \int_t^{t+\tau} z(s)ds   \leqslant a_3,
$$
where $a_i(i=1,2,3)$ are positive constants. Then
$$
z(t) \leqslant \left(\frac{a_3}{\tau}+a_2\right) e^{a_1} \quad  \text { for all } t \in\left(0, T\right).
$$
\end{lem}

\begin{lem}[{\cite[Lemma 2.3]{2018-MMMAS-PangWang}}]\label{lem-0926-3}
Let $T \in(0, \infty]$, $\tau \in\left(0, T\right)$ and suppose that $z$ is a non-negative absolutely continuous function satisfying
$$
z^{\prime}(t)+a(t) z(t) \leqslant b(t) z(t)+c(t) \quad \text { for a.e. } t \in(0, T)
$$
with some functions $a(t)>0$, $b(t) \geqslant 0$, $c(t) \geqslant 0$, and $a,b,c \in L_{\text {loc }}^1(0, T)$ for which there exist $b_1>0$, $c_1>0$ and $\varrho>0$ such that
$$
\sup_{0 \leqslant t \leqslant T-\tau} \int_t^{t+\tau} b(s) d s \leqslant b_1, \quad \sup _{0 \leqslant t \leqslant T-\tau} \int_t^{t+\tau} c(s) d s \leqslant c_1
$$
and
$$
\int_t^{t+\tau} a(s) d s-\int_t^{t+\tau} b(s) d s \geqslant \varrho \quad \text { for any } t \in(0, T-\tau).
$$
Then
$$
z(t) \leqslant z(0) e^{b_1}+\frac{c_1 e^{2 b_1}}{1-e^{-\varrho}}+c_1 e^{b_1} \quad \text { for all } t \in(0, T).
$$
\end{lem}


\section{Local existence and basic properties of solutions.}\label{section3}

To construct a weak solution to the system \eqref{SYS:MAIN}, we consider the following regularized problem for $\varepsilon \in(0,1)$:
\begin{align}\label{sys-regul}
\begin{cases}
u_{\varepsilon t}=\nabla \cdot\left(u_{\varepsilon} v_{\varepsilon} \nabla u_{\varepsilon}\right)
  - \nabla \cdot\left(u^{2}_{\varepsilon} v_{\varepsilon} \nabla v_{\varepsilon}\right)+ u_{\varepsilon} - u_{\varepsilon}^2, 
    & x \in \Omega, t>0, \\ 
v_{\varepsilon t}=\Delta v_{\varepsilon}-u_{\varepsilon} v_{\varepsilon}, 
   & x \in \Omega, t>0, \\ 
\frac{\partial u_{\varepsilon}}{\partial \nu}=\frac{\partial v_{\varepsilon}}{\partial \nu}=0, 
   & x \in \partial \Omega, t>0, \\ 
u_{\varepsilon}(x, 0)=u_{0 \varepsilon}(x)=u_0(x)+\varepsilon, \quad 
   v_{\varepsilon}(x, 0)=v_{0 \varepsilon}(x)=v_0(x), 
   & x \in \Omega.
\end{cases}
\end{align}
 
Using Amann’s theory \cite{1989-MZ-Amann}, one can derive the following local existence result; see also \cite[Lemma 2.1]{2022-NARWA-Winkler}. For convenience, we first state the local existence result for classical solutions to \eqref{sys-regul}.

\begin{lem}\label{lemma-2.1}
Let $\Omega \subset \mathbb{R}^2$ be a bounded domain with smooth boundary, and suppose that (\ref{assIniVal}) holds. Then for each $\varepsilon \in(0,1)$, there exist $T_{\max, \varepsilon} \in(0, \infty]$ and at least one pair $\left(u_{\varepsilon}, v_{\varepsilon}\right)$ of functions
\begin{align}\label{-2.6}
\begin{cases}
u_{\varepsilon} \in C^0\left(\overline{\Omega} \times\left[0, T_{\max, \varepsilon}\right)\right) \cap C^{2,1}\left(\overline{\Omega} \times\left(0, T_{\max, \varepsilon}\right)\right) \quad \text{ and } \\
v_{\varepsilon} \in \cap_{q > 2} C^0\left(\left[0, T_{\max, \varepsilon}\right); W^{1, q}(\Omega)\right) \cap C^{2,1}\left(\overline{\Omega} \times\left(0, T_{\max, \varepsilon}\right)\right)
\end{cases}
\end{align}
which are such that $u_{\varepsilon}>0$ and $v_{\varepsilon}>0$ in $\overline{\Omega} \times\left[0, T_{\max, \varepsilon}\right)$, that $\left(u_{\varepsilon}, v_{\varepsilon}\right)$ solves \eqref{sys-regul} in the classical sense, and that
\begin{align}\label{-2.7}
if \text{  } T_{\max, \varepsilon}<\infty, \quad \text { then } \quad \limsup _{t \nearrow T_{\max, \varepsilon}}\left\|u_{\varepsilon}(t)\right\|_{L^{\infty}(\Omega)}=\infty.
\end{align}
In addition, this solution satisfies
\begin{align}\label{-2.9}
\left\|v_{\varepsilon}(t)\right\|_{L^{\infty}(\Omega)} \leqslant \|v_0\|_{L^{\infty}(\Omega)}, \quad \text { for all } t \in\left(0, T_{\max, \varepsilon}\right)
\end{align}
\end{lem}

Next, we present some basic properties of solutions to the approximate problem \eqref{sys-regul}.


\begin{lem}\label{lem-1st-est}
Assume that \eqref{assIniVal} holds. Then there exists a positive constant $m$ such that
\begin{align}\label{-3.4}
\int_{\Omega} u_{\varepsilon} \leqslant m \quad \text { for all } t \in\left(0, T_{\max , \varepsilon}\right) \text { and } \varepsilon \in(0,1).
\end{align}
\end{lem}

\begin{proof}
Let $\varepsilon \in(0,1)$. By Hölder's inequality, we have $\left(\int_{\Omega} u_{\varepsilon}\right)^2 \leqslant\left(\int_{\Omega} u_{\varepsilon}^2\right)|\Omega|$. An integration of the first equation in \eqref{sys-regul} yields 
\begin{align}\label{eq:mass01}
\frac{d}{dt}\int_{\Omega} u_{\varepsilon} & =  \int_{\Omega} u_{\varepsilon}- \int_{\Omega} u_{\varepsilon}^2 \nonumber\\
& \leqslant  \int_{\Omega} u_{\varepsilon}-\frac{1}{|\Omega|}\left(\int_{\Omega} u_{\varepsilon}\right)^2 \quad \text { for all } t \in\left(0, T_{\max , \varepsilon}\right).
\end{align}
Then applying the ODE comparison to \eqref{eq:mass01}, we obtain 
$$
\int_{\Omega} u_{\varepsilon}(t) \leqslant m :=|\Omega|+ \int_{\Omega} u_0(x)  \quad \text { for all } t \in\left(0, T_{\max , \varepsilon}\right).  
$$  
We complete the proof.
\end{proof}

\begin{lem}\label{lem-1st-est1}
Assume that \eqref{assIniVal} holds. Then there exists constant $C>0$, independent of $t$ and $\varepsilon$, such that 
\begin{align}\label{-3.4ss}
\int_{\Omega} v_{\varepsilon} \leqslant C\quad \text { for all } t \in\left(0, T_{\max , \varepsilon}\right) \text { and } \varepsilon \in(0,1)
\end{align}
and
\begin{align}\label{-2.10}
\int_{0}^{T_{\max, \varepsilon}}\!\!\! \int_{\Omega} u_{\varepsilon} v_{\varepsilon} \leqslant C \quad \text { for all } \varepsilon \in(0,1).
\end{align}
\end{lem}

\begin{proof}
Let $\varepsilon \in(0,1)$. An integration of the second equation in \eqref{sys-regul} yields 
\begin{align}\label{eq:mass02}
\frac{d}{dt}\int_{\Omega} v_{\varepsilon} & =  -\int_{\Omega} u_{\varepsilon} v_{\varepsilon}
 \quad \text { for all } t \in\left(0, T_{\max, \varepsilon}\right).
\end{align}
By direct integration of \eqref{eq:mass02}, we obtain 
\begin{align*}
\int_{\Omega} v_{\varepsilon} +\int_0^{t} \int_{\Omega} u_{\varepsilon} v_{\varepsilon} =  \int_{\Omega} v_{0 } \quad \text { for all } t \in\left(0, T_{\max , \varepsilon}\right).
\end{align*}
This completes the proof of \eqref{-3.4ss} and \eqref{-2.10} due to \eqref{-2.9}. 
\end{proof}

\begin{lem}\label{lem-1st-est2}
Assume that \eqref{assIniVal} holds. Then there exists constant $C>0$, independent of $t$ and $\varepsilon$, such that 
\begin{align}\label{-3.5}
\int_t^{t+\tau} \int_{\Omega} u^2_{\varepsilon}\leqslant C  \quad \text { for all } t \in\left(0,  \widetilde{T}_{\max,\varepsilon } \right) \text { and } \varepsilon \in(0,1).
\end{align}
where
\begin{align}\label{-3.5ss}
\tau:=\min \left\{1, \frac{1}{2} T_{\max,\varepsilon}\right\} \quad \text { and } \quad \widetilde{T}_{\max }:=\left\{\begin{array}{lll}
T_{\max,\varepsilon }-\tau & \text { if } & T_{\max,\varepsilon }<\infty, \\
\infty & \text { if } & T_{\max,\varepsilon }=\infty.
\end{array}\right.
\end{align}
\end{lem}

\begin{proof}
The estimate
\begin{align*}
\int_t^{t+\tau} \int_{\Omega} u_{\varepsilon}^2 & \leqslant  \int_t^{t+\tau} \int_{\Omega} u_{\varepsilon}+ \int_{\Omega} u_{\varepsilon}(\tau)- \int_{\Omega} u_{\varepsilon}(t+\tau) \\
& \leqslant  2 m \quad \text { for all } t \in(0, \widetilde{T}_{\max, \varepsilon }) \text { and } \varepsilon \in(0,1),
\end{align*}
results from \eqref{eq:mass01} and \eqref{-3.5ss} after time-integration.
\end{proof}

\begin{lem}\label{lem-0927-1}
Assume that \eqref{assIniVal} holds. Then there exists constant $C>0$, independent of $t$ and $\varepsilon$, such that 
\begin{align}\label{0927-12}
\int_{\Omega} u_{\varepsilon} \ln u_{\varepsilon} \leqslant C  \quad \text { for all } t \in\left(0, T_{\max, \varepsilon} \right) \text { and } \varepsilon \in(0,1)
\end{align}
and
\begin{align}\label{0926-12}
\int_t^{t+\tau} \int_{\Omega} \frac{|\nabla v_{\varepsilon}|^4}{v_{\varepsilon}^3}  \leqslant C  \quad \text { for all } t \in\left(0,  \widetilde{T}_{\max, \varepsilon } \right) \text { and } \varepsilon \in(0,1)
\end{align}
as well as
\begin{align}\label{0926-13}
\int_t^{t+\tau} \int_{\Omega} \frac{u_{\varepsilon}}{v_{\varepsilon}} |\nabla v_{\varepsilon }|^2 \leqslant C  \quad \text { for all } t \in\left(0,  \widetilde{T}_{\max, \varepsilon } \right) \text { and } \varepsilon \in(0,1),
\end{align}
where $\tau$ and $\widetilde{T}_{\max, \varepsilon }$ are defined by \eqref{-3.5ss}.
\end{lem}

\begin{proof}
In view of \eqref{sys-regul} and the Cauchy-Schwarz inequality, we infer that
\begin{align*}
& \frac{d}{d t}\left\{\int_{\Omega} u_{\varepsilon} \ln u_{\varepsilon}-\int_{\Omega} u_{\varepsilon} v_{\varepsilon}\right\}+\int_{\Omega} v_{\varepsilon}\left|\nabla u_{\varepsilon}\right|^2 +\int_{\Omega} u_{\varepsilon}^2 v_{\varepsilon}\left|\nabla v_{\varepsilon}\right|^2 \nonumber\\
\leqslant & 2 \int_{\Omega} u_{\varepsilon} v_{\varepsilon} \nabla u_{\varepsilon} \cdot \nabla v_{\varepsilon} + \int_{\Omega} \nabla u_{\varepsilon} \cdot \nabla v_{\varepsilon}+2 \int_{\Omega} u_{\varepsilon}^2 v_{\varepsilon} +  \int_{\Omega} u_{\varepsilon} \nonumber\\
&  +\int_{\Omega} u_{\varepsilon} (\ln u_{\varepsilon}-u_{\varepsilon} \ln u_{\varepsilon} )-\int_{\Omega} u_{\varepsilon}^2\nonumber\\
\leqslant & \int_{\Omega} v_{\varepsilon}\left|\nabla u_{\varepsilon}\right|^2 +\int_{\Omega} u_{\varepsilon}^2 v_{\varepsilon}\left|\nabla v_{\varepsilon}\right|^2 + \int_{\Omega} \nabla u_{\varepsilon} \cdot \nabla v_{\varepsilon}+2 \int_{\Omega} u_{\varepsilon}^2 v_{\varepsilon} \nonumber\\
& + \int_{\Omega} u_{\varepsilon} +\int_{\Omega} u_{\varepsilon} (\ln u_{\varepsilon}-u_{\varepsilon} \ln u_{\varepsilon} )-\int_{\Omega} u_{\varepsilon}^2
\end{align*}
for all $t \in(0, T_{\max, \varepsilon})$ and $\varepsilon \in(0,1)$. Using $\ln \xi -\xi \ln \xi \leqslant 0$ for all $\xi > 0$ and neglecting a nonpositive summand  yields that
\begin{align}\label{0925-11}
\frac{d}{d t}\left\{\int_{\Omega} u_{\varepsilon} \ln u_{\varepsilon}-\int_{\Omega} u_{\varepsilon} v_{\varepsilon}\right\}
\leqslant   \int_{\Omega} \nabla u_{\varepsilon} \cdot \nabla v_{\varepsilon}+2 \int_{\Omega} u_{\varepsilon}^2 v_{\varepsilon} +  \int_{\Omega} u_{\varepsilon}
\end{align}
for all $t \in(0, T_{\max, \varepsilon})$ and $\varepsilon \in(0,1)$. According to the second equation in the system \eqref{sys-regul} we compute
\begin{align}\label{-3.11}
\frac{1}{2} \dt \int_{\Omega} \frac{\left|\nabla v_{\varepsilon}\right|^2}{v_{\varepsilon}}
= & \int_{\Omega} \frac{1}{v_{\varepsilon}} \nabla v_{\varepsilon} \cdot \nabla\left\{\Delta v_{\varepsilon}
    -u_{\varepsilon} v_{\varepsilon}\right\}
    -\frac{1}{2} \int_{\Omega} \frac{1}{v_{\varepsilon}^2}\left|\nabla v_{\varepsilon}\right|^2 \cdot\left\{\Delta v_{\varepsilon}-u_{\varepsilon} v_{\varepsilon}\right\} \nonumber\\
= & \int_{\Omega} \frac{1}{v_{\varepsilon}} \nabla v_{\varepsilon} \cdot \nabla\Delta v_{\varepsilon}
    -\int_{\Omega} \frac{u_{\varepsilon}}{v_{\varepsilon}}\left|\nabla v_{\varepsilon}\right|^2
    -\int_{\Omega} \nabla u_{\varepsilon} \cdot \nabla v_{\varepsilon}\nonumber\\
  & -\frac{1}{2} \int_{\Omega} \frac{1}{v_{\varepsilon}^2}\left|\nabla v_{\varepsilon}\right|^2 \cdot\Delta v_{\varepsilon}
    +\frac{1}{2} \int_{\Omega} \frac{u_{\varepsilon}}{v_{\varepsilon}}\left|\nabla v_{\varepsilon}\right|^2\nonumber\\
= & \int_{\Omega} \frac{1}{v_{\varepsilon}} \nabla v_{\varepsilon} \cdot \nabla\Delta v_{\varepsilon}
    -\frac{1}{2} \int_{\Omega} \frac{1}{v_{\varepsilon}^2}\left|\nabla v_{\varepsilon}\right|^2 \cdot\Delta v_{\varepsilon}\nonumber\\
  & -\frac{1}{2} \int_{\Omega} \frac{u_{\varepsilon}}{v_{\varepsilon}}\left|\nabla v_{\varepsilon}\right|^2
    -\int_{\Omega} \nabla u_{\varepsilon} \cdot \nabla v_{\varepsilon}
\end{align}
for all $t \in\left(0, T_{\max, \varepsilon}\right)$ and $\varepsilon \in(0,1)$. 
Based on \cite[Lemma~3.2]{2012-CPDE-Winkler}, we have the integral identity
\begin{align*}
\int_{\Omega} \frac{1}{v_{\varepsilon}} \nabla v_{\varepsilon} \cdot \nabla \Delta v_{\varepsilon}
-\frac{1}{2} \int_{\Omega} \frac{1}{v_{\varepsilon}^2}|\nabla v_{\varepsilon}|^2 \Delta v_{\varepsilon}
= -\int_{\Omega} v_{\varepsilon}\left|D^2 \ln v_{\varepsilon}\right|^2
  +\frac{1}{2} \int_{\partial \Omega} \frac{1}{v_{\varepsilon}} \frac{\partial|\nabla v_{\varepsilon}|^2}{\partial \nu}.
\end{align*}
Therefore, \eqref{-3.11} becomes
\begin{align}\label{-3.11-2}
&\frac{1}{2} \dt \int_{\Omega} \frac{\left|\nabla v_{\varepsilon}\right|^2}{v_{\varepsilon}}
  +\int_{\Omega} v_{\varepsilon}\left|D^2 \ln v_{\varepsilon}\right|^2
    +\frac{1}{2} \int_{\Omega} \frac{u_{\varepsilon}}{v_{\varepsilon}}\left|\nabla v_{\varepsilon}\right|^2\nonumber\\
   =&  \frac{1}{2} \int_{\partial \Omega} \frac{1}{v_{\varepsilon}} \frac{\partial|\nabla v_{\varepsilon}|^2}{\partial \nu}
    -\int_{\Omega} \nabla u_{\varepsilon} \cdot \nabla v_{\varepsilon}
\end{align}
for all $t \in\left(0, T_{\max, \varepsilon}\right)$ and $\varepsilon \in(0,1)$. By \cite[Lemma 3.3]{2012-CPDE-Winkler} and \cite[Lemma 3.4]{2022-NARWA-Winkler}, we see that
\begin{equation}\label{-3.12}
\int_{\Omega} v_{\varepsilon}\left|D^2 \ln v_{\varepsilon}\right|^2 
\geqslant c_1 \int_{\Omega} \frac{\left|D^2 v_{\varepsilon}\right|^2}{v_{\varepsilon}}
  +c_1 \int_{\Omega} \frac{\left|\nabla v_{\varepsilon}\right|^4}{v_{\varepsilon}^3}
\end{equation}
where $c_1=\frac{1}{14+8\sqrt{2}}$.
Thanks to \cite[Lemma 4.2]{2014-AIHPCANL-MizoguchiSouplet} and a boundary trace embedding inequality (cf. \cite[Theorem 1, p.272]{2010--Evans}), we have  
$$
\frac{\partial|\nabla v_{\varepsilon}|^2}{\partial \nu} \leqslant c_2|\nabla v_{\varepsilon}|^2 
$$
and
$$
\int_{\partial \Omega}|v_{\varepsilon}| \leqslant c_3 \int_{\Omega}|\nabla v_{\varepsilon}|+c_3 \int_{\Omega}|v_{\varepsilon}|,
$$
where $c_2=c_2(\Omega)>0$ denotes an upper bound for the curvatures of $\partial \Omega$ and $c_3$ is a constant depending only on $\Omega$. Therefore, by Young's inequality, we obtain
\begin{align}\label{n-3.13}
\frac{1}{2}\int_{\partial \Omega} \frac{1}{v_{\varepsilon}} \cdot \frac{\partial\left|\nabla v_{\varepsilon}\right|^2}{\partial \nu} 
& \leqslant \frac{c_2}{2} \int_{\partial \Omega} \frac{\left|\nabla v_{\varepsilon}\right|^2}{v_{\varepsilon}} \nonumber\\
& \leqslant \frac{c_2 c_3}{2} \int_{\Omega}\left|\nabla\Big(\frac{\left|\nabla v_{\varepsilon}\right|^2}{v_{\varepsilon}}\Big)\right|
    +\frac{c_2 c_3}{2} \int_{\Omega} \frac{\left|\nabla v_{\varepsilon}\right|^2}{v_{\varepsilon}} \nonumber\\
& \leqslant c_2 c_3 \int_{\Omega} \frac{\left|D^2 v_{\varepsilon} \cdot \nabla v_{\varepsilon}\right|}{v_{\varepsilon}}
    +\frac{c_2 c_3}{2} \int_{\Omega} \frac{\left|\nabla v_{\varepsilon}\right|^3}{v_{\varepsilon}^2}
    +\frac{c_2 c_3}{2} \int_{\Omega} \frac{\left|\nabla v_{\varepsilon}\right|^2}{v_{\varepsilon}} \nonumber\\
&\leqslant c_1\int_{\Omega} \frac{\left|D^2 v_{\varepsilon}\right|^2}{v_{\varepsilon}}
  +\frac{c_1}{4} \int_{\Omega} \frac{\left|\nabla v_{\varepsilon}\right|^4}{v_{\varepsilon}^3}
  +c_4 \int_{\Omega} \frac{\left|\nabla v_{\varepsilon}\right|^2}{v_{\varepsilon}}\nonumber\\
&\leqslant c_1\int_{\Omega} \frac{\left|D^2 v_{\varepsilon}\right|^2}{v_{\varepsilon}}
  +\frac{c_1}{2} \int_{\Omega} \frac{\left|\nabla v_{\varepsilon}\right|^4}{v_{\varepsilon}^3}
  +c_5 \int_{\Omega} v_{\varepsilon},
\end{align}
where $c_4= \frac{c^2_2 c^2_3}{2 c_1}+\frac{c_2 c_3}{2}$ and $c_5=\frac{c^2_4}{c_1}$. Summing up \eqref{-3.11-2}-\eqref{n-3.13}, we conclude that
\begin{align}\label{-3.14}
\frac{1}{2}\frac{d}{dt} \int_{\Omega} \frac{|\nabla v_{\varepsilon}|^2}{v_{\varepsilon}}
+ \frac{c_1}{2}\int_{\Omega} \frac{|\nabla v_{\varepsilon}|^4}{v_{\varepsilon}^3}
+ \frac{1}{2} \int_{\Omega} \frac{u_{\varepsilon}}{v_{\varepsilon}} |\nabla v_{\varepsilon }|^2 
\leqslant  -\int_{\Omega} \nabla u_{\varepsilon} \cdot \nabla v_{\varepsilon} +  c_5 \int_{\Omega} v_{\varepsilon}
\end{align}
for all $t \in(0, T_{\max, \varepsilon})$ and $\varepsilon \in(0,1)$. Combining \eqref{0925-11} with  \eqref{-3.14}, we get 
\begin{align*}
\frac{d}{d t}\left\{\int_{\Omega} u_{\varepsilon} \ln u_{\varepsilon}-\int_{\Omega} u_{\varepsilon} v_{\varepsilon}+\frac{1}{2} \int_{\Omega} \frac{|\nabla v_{\varepsilon}|^2}{v_{\varepsilon}} \right\} & + \frac{c_1}{2}\int_{\Omega} \frac{|\nabla v_{\varepsilon}|^4}{v_{\varepsilon}^3} + \frac{1}{2} \int_{\Omega} \frac{u_{\varepsilon}}{v_{\varepsilon}} |\nabla v_{\varepsilon }|^2\nonumber\\
& \leqslant  2 \int_{\Omega} u_{\varepsilon}^2 v_{\varepsilon} +  \int_{\Omega} u_{\varepsilon} +  c_5 \int_{\Omega} v_{\varepsilon}
\end{align*}
for all $t \in(0, T_{\max, \varepsilon})$ and $\varepsilon \in(0,1)$. Letting $y_{\varepsilon}(t):= \int_{\Omega} u_{\varepsilon} \ln u_{\varepsilon}-\int_{\Omega} u_{\varepsilon} v_{\varepsilon}+\frac{1}{2} \int_{\Omega} \frac{|\nabla v_{\varepsilon}|^2}{v_{\varepsilon}}$, adding $y_{\varepsilon}(t)$ to both sides of the above equation and using the estimates $\int_{\Omega} u_{\varepsilon} \ln u_{\varepsilon} \leqslant \int_{\Omega} u_{\varepsilon}^2$, $\frac{1}{2} \int_{\Omega} \frac{|\nabla v_{\varepsilon}|^2}{v_{\varepsilon}} \leqslant \frac{c_1}{4}\int_{\Omega} \frac{|\nabla v_{\varepsilon}|^4}{v_{\varepsilon}^3}+ \frac{1}{4 c_1}\int_{\Omega} v_{\varepsilon }$ and \eqref{-2.9}, we obtain 
\begin{align}\label{0926-11}
y_{\varepsilon}^{\prime}(t)+y_{\varepsilon}(t)+ \frac{c_1}{4}\int_{\Omega} \frac{|\nabla v_{\varepsilon}|^4}{v_{\varepsilon}^3} & + \frac{1}{2} \int_{\Omega} \frac{u_{\varepsilon}}{v_{\varepsilon}} |\nabla v_{\varepsilon }|^2\nonumber\\
& \leqslant 2 \int_{\Omega} u_{\varepsilon}^2 v_{\varepsilon} +  \int_{\Omega} u_{\varepsilon} +  c_6 \int_{\Omega} v_{\varepsilon} + \int_{\Omega} u_{\varepsilon}^2\nonumber\\
& \leqslant (2\|v_0\|_{L^{\infty}(\Omega)} +1 )\int_{\Omega} u_{\varepsilon}^2 +  \int_{\Omega} u_{\varepsilon} +   c_6 \int_{\Omega} v_{\varepsilon}\nonumber\\
&\quad \text { for all } t \in\left(0, T_{\max } \right) \text { and } \varepsilon \in(0,1),
\end{align}
where $c_6=c_5+\frac{1}{4 c_1}$. By Lemmas \ref{lem-1st-est}, \ref{lem-1st-est1} and \ref{lem-1st-est2}, there exist a constant $c_7>0$, independent of $t$ and $\varepsilon$, such that 
\begin{align*}
(2\|v_0\|_{L^{\infty}(\Omega)} +1 ) \int_t^{t+\tau} \int_{\Omega} u^2_{\varepsilon}+ \int_t^{t+\tau} \int_{\Omega} u_{\varepsilon} + c_6 \int_t^{t+\tau} \int_{\Omega} v_{\varepsilon} \leqslant c_7.
\end{align*}
for all $t \in(0, \widetilde{T}_{\max, \varepsilon})$ and $\varepsilon \in(0,1)$. Consequently, applying Lemma \ref{lem-0926-1} to \eqref{0926-11} implies the existence of another constant $c_8>0$ also independent of $t$ and $\varepsilon$, satisfying
\begin{align*}
\int_{\Omega} u_{\varepsilon} \ln u_{\varepsilon}-\int_{\Omega} u_{\varepsilon} v_{\varepsilon}& +\frac{1}{2} \int_{\Omega} \frac{|\nabla v_{\varepsilon}|^2}{v_{\varepsilon}}\nonumber\\
& \leqslant c_8  :=\max \left\{(u_{0}+1)\ln (u_{0}+1)+\frac{1}{2} \int_{\Omega} \frac{|\nabla v_{0}|^2}{v_{0}}+c_7,~\frac{c_7}{\tau}+2 c_7\right\}.
\end{align*}
This follows from \eqref{-2.9} and \eqref{-3.4} that
\begin{align}\label{0108-11}
-\frac{|\Omega|}{e}-m\|v_0\|_{L^\infty(\Omega)}
\leqslant
\int_{\Omega} u_{\varepsilon} \ln u_{\varepsilon}
-\int_{\Omega} u_{\varepsilon} v_{\varepsilon}
\leqslant
y_{\varepsilon}(t)
\leqslant
c_8
\end{align}
for all $t \in (0, T_{\max,\varepsilon})$ and $\varepsilon \in (0,1)$.
Then, integrating both sides of \eqref{0926-11} over the interval $(t,t+\tau)$ and using \eqref{0108-11} and $0<\tau<1$, we obtain
\begin{align*}
\frac{c_1}{4} \int_t^{t+\tau} \int_{\Omega} \frac{|\nabla v_{\varepsilon}|^4}{v_{\varepsilon}^3}  + \frac{1}{2} \int_t^{t+\tau} \int_{\Omega} \frac{u_{\varepsilon}}{v_{\varepsilon}} |\nabla v_{\varepsilon }|^2  
&  \leqslant c_7+c_8+2\frac{|\Omega|}{e}+2 m  \|v_0\|_{L^{\infty}(\Omega)}.
\end{align*}
for all $t \in(0, \widetilde{T}_{\max })$ and $\varepsilon \in(0,1)$. Thus, we complete the proof of \eqref{0927-12}, \eqref{0926-12} and \eqref{0926-13}.
\end{proof}



\section{Uniform $L^p$ Estimates for $u_{\varepsilon}$ with $p>1$}\label{sect-4}

We now present our two key tools, which furnish the foundation for a subsequent $L^p$-regularity argument for $u_{\varepsilon}$.

\begin{lem}\label{lemma-3.4}
Let $\Omega \subset \mathbb{R}^2$ be a bounded domain with smooth boundary and $p \geqslant 1$. For any $\varphi, \psi \in C^1(\overline{\Omega})$ satisfying $\varphi,\psi>0$ in $\overline{\Omega}$, there holds
\begin{align}\label{eq-6.1}
\int_{\Omega} \varphi^{p +1} \psi 
  \leqslant  c \left\{\int_{\Omega} \psi |\nabla \varphi|^2 
  +\int_{\Omega} \frac{\varphi}{\psi}|\nabla \psi|^2\
    +\int_{\Omega} \varphi \psi  \right\}\cdot \int_{\Omega} \varphi^p + c \int_{\Omega} \psi |\nabla \varphi|^2 
\end{align}
for some constant $c=c(p,\Omega)>0$.
\end{lem}
\begin{proof}
The Sobolev embedding inequality in $\Omega \subset \mathbb{R}^2$ yields $c_1(\Omega)>0$ fulfilling
\begin{align}\label{eq-6.2}
\int_{\Omega} \rho^2 \leqslant c_1(\Omega) \|\nabla \rho\|_{L^1(\Omega)}^2 + c_1(\Omega) \|\rho\|_{L^{1}(\Omega)}^2, \quad 
  \rho \in W^{1,1}(\Omega).
\end{align}
For any $\varphi, \psi \in C^1(\overline{\Omega})$ satisfying $\varphi,\psi>0$ in $\overline{\Omega}$, we apply \eqref{eq-6.1} with $\rho=\varphi^{\frac{p +1}{2}} \psi^{\frac{1}{2}}$ to infer that
\begin{align}\label{eq-6.3}
\int_{\Omega} \varphi^{p +1} \psi \leqslant 
&\ c_1(\Omega) \left\{\int_{\Omega}\left|\frac{p+1}{2}\varphi^{\frac{p -1}{2}}\psi^{\frac{1}{2}}\nabla\varphi
    +\frac{1}{2}\varphi^{\frac{p +1}{2}} \psi^{-\frac{1}{2}} \nabla \psi\right|\right\}^2
 +c_1(\Omega) \cdot\left\{\int_{\Omega}\varphi^{\frac{p+1}{2}} \psi^{\frac{1}{2}}\right\}^{2}\nonumber\\
\leqslant &\ \frac{(p+1)^2 c_1(\Omega)}{2} \left\{\int_{\Omega} \varphi^{\frac{p -1}{2}} \psi^{\frac{1}{2}}|\nabla \varphi|\right\}^2
  +\frac{c_1(\Omega)}{2}\left\{\int_{\Omega} \varphi^{\frac{p +1}{2}} \psi^{-\frac{1}{2}}|\nabla \psi|\right\}^{2} \nonumber\\
& +c_1(\Omega)\left\{\int_{\Omega}\varphi^{\frac{p +1}{2}} \psi^{\frac{1}{2}}\right\}^{2}.
\end{align}
By Young's inequality, we have
\begin{align*}
\left\{\int_{\Omega} \varphi^{\frac{p -1}{2}} \psi^{\frac{1}{2}}|\nabla \varphi|\right\}^2 \leqslant \int_{\Omega} \varphi^{p-1}  \cdot \int_{\Omega} \psi |\nabla \varphi|^2 \leqslant \int_{\Omega} \varphi^{p}  \cdot \int_{\Omega} \psi |\nabla \varphi|^2 +|\Omega| \int_{\Omega} \psi |\nabla \varphi|^2
\end{align*}
and
\begin{align*}
\left\{\int_{\Omega} \varphi^{\frac{p +1}{2}} \psi^{-\frac{1}{2}}|\nabla \psi|\right\}^2 =
\left\{\int_{\Omega} \varphi^{\frac{p}{2}}\cdot \frac{\varphi^{\frac{1}{2}}}{\psi^{\frac{1}{2}}} |\nabla \psi|\right\}^2
\leqslant \int_{\Omega} \varphi^p  \cdot \int_{\Omega} \frac{\varphi}{\psi}|\nabla \psi|^2.
\end{align*}
Based on H\"{o}lder's inequality, we see that
\begin{align*}
\left\{\int_{\Omega}\varphi^{\frac{p +1}{2}} \psi^{\frac{1}{2}}\right\}^{2} & 
=\left\{\int_{\Omega} \varphi^{\frac{p }{2}}  \cdot (\varphi \psi)^{\frac{1}{2}}\right\}^{2} \leqslant \int_{\Omega} \varphi^p  \cdot \int_{\Omega} \varphi \psi.
\end{align*}
Thus, \eqref{eq-6.1} results from \eqref{eq-6.3} letting $c=c(p,\Omega)=\max \left\{\frac{(p +1)^2 c_1(\Omega)}{2}, \frac{(p +1)^2 |\Omega|c_1(\Omega)}{2},c_1(\Omega)\right\}$.
\end{proof}

\begin{lem}[{\cite[Lemma A.2]{2025-MMaMiAS-ZhangLi}}]\label{lemma-3.5}
Let $\Omega \subset \mathbb{R}^2$ be a bounded domain with smooth boundary and $p \geqslant 1$. For each $\eta>0$ and any $\varphi, \psi \in C^1(\overline{\Omega})$ satisfying $\varphi,\psi>0$ in $\overline{\Omega}$, there holds
\begin{align}\label{eq-6.4}
\int_\Omega \varphi^{p+1} \psi |\nabla \psi|^2
\leqslant & \eta \int_\Omega \varphi^{p-1} \psi|\nabla \varphi|^2
  + c \left\{\left\|\psi \right\|_{L^{\infty}(\Omega)}+\frac{\left\|\psi \right\|^3_{L^{\infty}(\Omega)}}{\eta}\right\} 
    \cdot \int_{\Omega} \varphi^{p+1} \psi  \cdot \int_\Omega \frac{|\nabla \psi|^4}{\psi^3} \nonumber\\
& + c \left\|\psi \right\|^2_{L^{\infty}(\Omega)} \cdot\left\{\int_\Omega \varphi\right\}^{2 p+1} 
    \cdot \int_\Omega \frac{|\nabla \psi|^4}{\psi^3} 
    + c \left\|\psi \right\|^2_{L^{\infty}(\Omega)} \cdot \int_{\Omega}\varphi \psi
\end{align}
for some constant $c=c(p)>0$.
\end{lem}

To establish a uniform bound for $\|u_{\varepsilon}(t)\|_{L^p(\Omega)}$, independent of $\varepsilon$ and valid for all $p>1$ and $t \in (0, T_{\max,\varepsilon})$, we derive a differential inequality for $\int_{\Omega} u_{\varepsilon}^p(t)$ by invoking Lemmas \ref{lemma-3.4} and \ref{lemma-3.5}.

\begin{lem}\label{lemma-3.9xx}
For $p > 1$, we have
\begin{align}\label{eq-3.1}
& \dt \int_{\Omega} u_{\varepsilon}^p 
+  \frac{p(p-1)}{4} \int_{\Omega} u_{\varepsilon}^{p-1} v_{\varepsilon} \left|\nabla 
u_{\varepsilon}\right|^2 + p  \int_{\Omega} u_{\varepsilon}^{p+1} \nonumber\\
\leqslant &   A \left\{
  \int_{\Omega}v_{\varepsilon}|\nabla u_{\varepsilon}|^2
  +\int_{\Omega} \frac{u_{\varepsilon}}{v_{\varepsilon}}|\nabla v_{\varepsilon}|^2
  +\int_{\Omega} u_{\varepsilon} v_{\varepsilon}  \right\} \cdot \int_{\Omega} u_{\varepsilon}^p  \cdot \int_\Omega \frac{|\nabla v_{\varepsilon}|^4}{v_{\varepsilon}^3} \nonumber\\ 
& +A \int_{\Omega} v_{\varepsilon} |\nabla u_{\varepsilon}|^2 \cdot \int_\Omega \frac{|\nabla v_{\varepsilon}|^4}{v_{\varepsilon}^3} + A \left\{\int_\Omega u_{\varepsilon} \right\}^{2 p+1} \cdot \int_\Omega \frac{|\nabla v_{\varepsilon}|^4}{v_{\varepsilon}^3}\nonumber\\ 
& + A \int_{\Omega} u_{\varepsilon} v_{\varepsilon}  + p \int_{\Omega}u_{\varepsilon}^{p}
\end{align}
for some positive constant $A$ independent of $\varepsilon$. 
\end{lem}

\begin{proof}
Multiplying the first equation of the system \eqref{sys-regul} by $u_{\varepsilon}^{p-1}$, integrating by parts and using Young's inequality, we obtain 
\begin{align}\label{eq-3.2}
\dt \int_{\Omega} u_{\varepsilon}^p
= & \ p \int_{\Omega} u_{\varepsilon}^{p-1}\left\{\nabla \cdot(u_{\varepsilon} v_{\varepsilon} \nabla u_{\varepsilon})- \nabla \cdot\left(u_{\varepsilon}^{2} v_{\varepsilon} \nabla v_{\varepsilon}\right)+  u_{\varepsilon}- u_{\varepsilon}^2\right\} \nonumber\\
= &-p(p-1) \int_{\Omega} u_{\varepsilon}^{p-1} v_{\varepsilon}\left|\nabla u_{\varepsilon}\right|^2+p (p-1) \int_{\Omega} u_{\varepsilon}^{p} v_{\varepsilon} \nabla u_{\varepsilon} \cdot \nabla v_{\varepsilon} + p  \int_{\Omega} u_{\varepsilon}^{p} - p  \int_{\Omega} u_{\varepsilon}^{p+1} \nonumber\\
\leqslant & -\frac{p(p-1)}{2} \int_{\Omega} u_{\varepsilon}^{p-1} v_{\varepsilon} \left|\nabla u_{\varepsilon}\right|^2
+\frac{p(p-1)}{2} \int_{\Omega} u_{\varepsilon}^{p+1} v_{\varepsilon}\left|\nabla v_{\varepsilon}\right|^2\nonumber\\
&  + p  \int_{\Omega} u_{\varepsilon}^{p} - p  \int_{\Omega} u_{\varepsilon}^{p+1}
\end{align}
for all $t \in(0, T_{\max, \varepsilon})$ and $\varepsilon \in(0,1)$. Letting $\eta=\frac{1}{2}$ in \eqref{eq-6.4} and using \eqref{-2.9}, we have
\begin{align*}
\frac{p(p-1)}{2} & \int_{\Omega} u_{\varepsilon}^{p+1} v_{\varepsilon}\left|\nabla v_{\varepsilon}\right|^2 
\leqslant \frac{p(p-1) }{4} \int_\Omega u_{\varepsilon}^{p-1} v_{\varepsilon}|\nabla u_{\varepsilon}|^2 
\nonumber\\
&+  \left\{ c_1   \left\|v_{0} \right\|_{L^{\infty}(\Omega)}
  +2c_1  \left\|v_{0} \right\|^3_{L^{\infty}(\Omega)}\right\} \cdot \int_{\Omega} u_{\varepsilon}^{p+1} v_{\varepsilon}  \cdot \int_\Omega \frac{|\nabla v_{\varepsilon}|^4}{v_{\varepsilon}^3} \nonumber\\
&+c_1  \left\|v_{0} \right\|^2_{L^{\infty}(\Omega)} \cdot\left\{\int_\Omega u_{\varepsilon} \right\}^{2 p+1} \cdot \int_\Omega \frac{|\nabla v_{\varepsilon}|^4}{v_{\varepsilon}^3} +  c_1 \left\|v_{0} \right\|^2_{L^{\infty}(\Omega)} \cdot \int_{\Omega} u_{\varepsilon} v_{\varepsilon} 
\end{align*}
for some constant $c_1=c_1(p,\Omega)>0$. Substituting this into \eqref{eq-3.2} yields 
\begin{align}\label{jia-2}
& \dt \int_{\Omega} u_{\varepsilon}^p 
+ \frac{p(p-1)}{4} \int_{\Omega} u_{\varepsilon}^{p-1} v_{\varepsilon} \left|\nabla 
u_{\varepsilon}\right|^2  + p  \int_{\Omega} u_{\varepsilon}^{p+1} \nonumber\\
\leqslant &   \left\{ c_1  \left\|v_{0} \right\|_{L^{\infty}(\Omega)}+2 c_1 \left\|v_{0} \right\|^3_{L^{\infty}(\Omega)}\right\} \cdot \int_{\Omega} u_{\varepsilon}^{p+1} v_{\varepsilon}  \cdot \int_\Omega \frac{|\nabla v_{\varepsilon}|^4}{v_{\varepsilon}^3} \nonumber\\
& +  c_1 \left\|v_{0} \right\|^2_{L^{\infty}(\Omega)} \cdot\left\{\int_\Omega u_{\varepsilon} \right\}^{2 p+1} \cdot \int_\Omega \frac{|\nabla v_{\varepsilon}|^4}{v_{\varepsilon}^3}\nonumber\\
& + c_1   \left\|v_{0} \right\|^2_{L^{\infty}(\Omega)} \cdot \int_{\Omega} u_{\varepsilon} v_{\varepsilon}  + p  \int_{\Omega}  u_{\varepsilon}^{p}
\end{align}
for all $t \in(0, T_{\max, \varepsilon})$ and $\varepsilon \in(0,1)$. It follows from \eqref{eq-6.1} that 
\begin{align}\label{jia-1}
\int_{\Omega} u_{\varepsilon}^{p+1} v_{\varepsilon} 
 \leqslant c_2 \left\{\int_{\Omega} v_{\varepsilon} |\nabla u_{\varepsilon}|^2
 +\int_{\Omega} \frac{u_{\varepsilon}}{v_{\varepsilon}}|\nabla v_{\varepsilon}|^2
  +\int_{\Omega} u_{\varepsilon} v_{\varepsilon}  \right\} \cdot \int_{\Omega} u_{\varepsilon}^p
  + c_2 \int_{\Omega} v_{\varepsilon} |\nabla u_{\varepsilon}|^2
\end{align}
for some constant $c_2=c_2(p, \Omega)>0$. 

Letting $A=\max \left\{c_1c_2  \left\|v_{0} \right\|_{L^{\infty}(\Omega)}+2 c_1c_2  \left\|v_{0} \right\|^3_{L^{\infty}(\Omega)}, c_1 \left\|v_{0} \right\|^2_{L^{\infty}(\Omega)}\right\}$ and combining \eqref{jia-2} with \eqref{jia-1}, we complete the proof.
\end{proof}

From  \eqref{-2.10}, \eqref{0926-12} and \eqref{0926-13}, we know that $\int_\Omega \frac{|\nabla v_{\varepsilon}|^4}{v_{\varepsilon}^3} + \frac{u_{\varepsilon}}{v_{\varepsilon}}|\nabla v_{\varepsilon}|^2+ u_{\varepsilon}v_{\varepsilon} \in L^1(t,t+\tau)$ for all $t \in(0, \widetilde{T}_{\max, \varepsilon})$ and $\varepsilon \in(0,1)$. However, this is not sufficient to solve the differential inequality \eqref{eq-3.1} for $\int_{\Omega} u_{\varepsilon}^p$. To overcome this difficulty, we aim to show that $\int_\Omega \frac{|\nabla v_{\varepsilon}|^4}{v_{\varepsilon}^3} \in(0, T_{\max, \varepsilon})$ and  $\int_{\Omega} v_{\varepsilon}\left|\nabla u_{\varepsilon}\right|^2 \in L^1(t,t+\tau)$. To this end, we derive a differential inequality for the following energy-like functional: $4b \int_{\Omega} u_{\varepsilon} \ln u_{\varepsilon}+\int_{\Omega} \frac{\left|\nabla v_{\varepsilon}\right|^4}{v_{\varepsilon}^3}$.

\begin{lem}\label{lemma-3.6}
There holds 
\begin{align}\label{0903-2036}
\dt\left\{4b \int_{\Omega} u_{\varepsilon} \ln u_{\varepsilon}+\int_{\Omega} \frac{\left|\nabla v_{\varepsilon}\right|^4}{v_{\varepsilon}^3}\right\} \leqslant & -b \int_{\Omega} v_{\varepsilon}\left|\nabla u_{\varepsilon}\right|^2- \int_{\Omega} u_{\varepsilon} v_{\varepsilon}^{-3}\left|\nabla v_{\varepsilon}\right|^4\nonumber\\
& - \int_{\Omega} v_{\varepsilon}^{-1}\left|\nabla v_{\varepsilon}\right|^2\left|D^2 \ln v_{\varepsilon}\right|^2\nonumber\\
& + 4b  \int_{\Omega} u_{\varepsilon}^{2} v_{\varepsilon}\left|\nabla v_{\varepsilon}\right|^2+
4 b  \int_{\Omega} u_{\varepsilon} + c \int_{\Omega} v_{\varepsilon}
\end{align}
for all $t \in(0, T_{\max, \varepsilon})$ and $\varepsilon \in(0,1)$, where $b$ and $c=c(b,\Omega)$ are some positive constants independent of $\varepsilon$.
\end{lem}

\begin{proof}
According to\cite[Lemma 3.4]{2022-DCDSSB-Winkler}, there exists a constant $b>0$ independent of $\varepsilon$ such that
\begin{align}\label{-3.21}
\int_{\Omega} \frac{|\nabla v_{\varepsilon}|^6}{v_{\varepsilon}^5} \leqslant b \int_{\Omega} v_{\varepsilon}^{-1}|\nabla v_{\varepsilon}|^2\left|D^2 \ln v_{\varepsilon}\right|^2
\end{align}
and
\begin{align}\label{-3.21a}
\int_{\Omega} v_{\varepsilon}^{-3}|\nabla v_{\varepsilon}|^{2}\left|D^2 v_{\varepsilon}\right|^2 \leqslant b \int_{\Omega} v_{\varepsilon}^{-1}|\nabla v_{\varepsilon}|^{2}\left|D^2 \ln v_{\varepsilon}\right|^2.
\end{align}
We can from \cite[Lemma 2.3]{2022-JDE-Li} obtain 
\begin{align}\label{-3.22}
\frac{d}{d t} \int_{\Omega} \frac{\left|\nabla v_{\varepsilon}\right|^4}{v_{\varepsilon}^3} + & 4 \int_{\Omega} v_{\varepsilon}^{-1}\left|\nabla v_{\varepsilon}\right|^2\left|D^2 \ln v_{\varepsilon}\right|^2  +\int_{\Omega} u_{\varepsilon} v_{\varepsilon}^{-3}\left|\nabla v_{\varepsilon}\right|^4  \nonumber\\
\leqslant & -4 \int_{\Omega} v_{\varepsilon}^{-2}\left|\nabla v_{\varepsilon}\right|^2\left(\nabla u_{\varepsilon} \cdot \nabla v_{\varepsilon}\right)
+ 2 \int_{\partial \Omega} v_{\varepsilon}^{-3}\left|\nabla v_{\varepsilon}\right|^2 \frac{\partial \left|\nabla v_{\varepsilon}\right|^2}{\partial \nu}
\end{align}
for all $t \in(0, T_{\max, \varepsilon})$ and $\varepsilon \in(0,1)$. Invoking Young's inequality and \eqref{-3.21} yields that
\begin{align}\label{-3.23}
-4 \int_{\Omega} v_{\varepsilon}^{-2}\left|\nabla v_{\varepsilon}\right|^2\left(\nabla u_{\varepsilon} \cdot \nabla v_{\varepsilon}\right) & \leqslant \frac{2}{b} \int_{\Omega} \frac{\left|\nabla v_{\varepsilon}\right|^6}{v_{\varepsilon}^5}+2 b \int_{\Omega} v_{\varepsilon}\left|\nabla u_{\varepsilon}\right|^2 \nonumber\\
& \leqslant 2 \int_{\Omega} v_{\varepsilon}^{-1}\left|\nabla v_{\varepsilon}\right|^2\left|D^2 \ln v_{\varepsilon}\right|^2 +2 b \int_{\Omega} v_{\varepsilon}\left|\nabla u_{\varepsilon}\right|^2.
\end{align}
In light of \eqref{-3.21}-\eqref{-3.21a} and \cite[Lemma 3.5]{2022-DCDSSB-Winkler}, one can find a constant $c=c(b,\Omega)>0$ such that
\begin{align}\label{-3.23a}
2 \int_{\partial \Omega} v_{\varepsilon}^{-3}|\nabla v_{\varepsilon}|^{2} \cdot \frac{\partial|\nabla v_{\varepsilon}|^2}{\partial \nu} 
\leqslant & \frac{1}{2b} \int_{\Omega} v_{\varepsilon}^{-3}|\nabla v_{\varepsilon}|^{2}\left|D^2 v_{\varepsilon}\right|^2+\frac{1}{2b} \int_{\Omega} \frac{\left|\nabla v_{\varepsilon}\right|^6}{v_{\varepsilon}^5}+c \int_{\Omega} v_{\varepsilon}\nonumber\\
\leqslant &  \int_{\Omega} v_{\varepsilon}^{-1}\left|\nabla v_{\varepsilon}\right|^2\left|D^2 \ln v_{\varepsilon}\right|^2  + c \int_{\Omega} v_{\varepsilon}. 
\end{align}
Multiplying the first equation in \eqref{sys-regul} by $1+\ln u_{\varepsilon}$, using $\ln \xi -\xi \ln \xi \leqslant 0$ for all $\xi > 0$ and neglecting a nonpositive summand, we use Cauchy-Schwarz inequality to infer that
\begin{align}\label{-3.24}
\dt \int_{\Omega} u_{\varepsilon} \ln u_{\varepsilon}&+\int_{\Omega} v_{\varepsilon}\left|\nabla u_{\varepsilon}\right|^2   =  \int_{\Omega} u_{\varepsilon} v_{\varepsilon} \nabla u_{\varepsilon} \cdot \nabla v_{\varepsilon}+ \int_{\Omega} u_{\varepsilon}  \ln u_{\varepsilon}\nonumber\\
& - \int_{\Omega} u_{\varepsilon}^2 \ln u_{\varepsilon} - \int_{\Omega} u_{\varepsilon}^2  + \int_{\Omega} u_{\varepsilon}\nonumber\\
& \leqslant   \frac{1}{4} \int_{\Omega} v_{\varepsilon}\left|\nabla u_{\varepsilon}\right|^2+ \int_{\Omega} u_{\varepsilon}^{2} v_{\varepsilon}\left|\nabla v_{\varepsilon}\right|^2 + \int_{\Omega} u_{\varepsilon}
\end{align}
for all $t \in(0, T_{\max, \varepsilon})$ and $\varepsilon \in(0,1)$. Gathering \eqref{-3.22}-\eqref{-3.24}, we conclude that
\begin{align*}
\dt\left\{4 b \int_{\Omega} u_{\varepsilon} \ln u_{\varepsilon}+\int_{\Omega} \frac{\left|\nabla v_{\varepsilon}\right|^4}{v_{\varepsilon}^3}\right\} 
\leqslant & -3 b \int_{\Omega} v_{\varepsilon}\left|\nabla u_{\varepsilon}\right|^2 +4 b \int_{\Omega} u_{\varepsilon}^{2} v_{\varepsilon}\left|\nabla v_{\varepsilon}\right|^2+4  b \int_{\Omega} u_{\varepsilon}  \nonumber\\
&-\int_{\Omega} u_{\varepsilon} v_{\varepsilon}^{-3}\left|\nabla v_{\varepsilon}\right|^4- \int_{\Omega} v_{\varepsilon}^{-1}\left|\nabla v_{\varepsilon}\right|^2\left|D^2 \ln v_{\varepsilon}\right|^2\nonumber\\
& +2 b \int_{\Omega} v_{\varepsilon}\left|\nabla u_{\varepsilon}\right|^2 + c \int_{\Omega} v_{\varepsilon} \nonumber\\
= & -b \int_{\Omega} v_{\varepsilon}\left|\nabla u_{\varepsilon}\right|^2- \int_{\Omega} u_{\varepsilon} v_{\varepsilon}^{-3}\left|\nabla v_{\varepsilon}\right|^4 \nonumber\\
& - \int_{\Omega} v_{\varepsilon}^{-1}\left|\nabla v_{\varepsilon}\right|^2\left|D^2 \ln v_{\varepsilon}\right|^2+ 4 b \int_{\Omega} u_{\varepsilon}^{2} v_{\varepsilon}\left|\nabla v_{\varepsilon}\right|^2 \nonumber\\
& + 4  b  \int_{\Omega} u_{\varepsilon}  + c \int_{\Omega} v_{\varepsilon}
\end{align*}
for all $t \in(0, T_{\max, \varepsilon})$ and $\varepsilon \in(0,1)$. The proof is complete.
\end{proof}

We are now in a position to derive $\varepsilon$- and $t$-independent estimates for 
$\int_{\Omega} u_{\varepsilon}^p$. To overcome this difficulty, we aim to show that $\int_\Omega \frac{|\nabla v_{\varepsilon}|^4}{v_{\varepsilon}^3} \in(0, T_{\max, \varepsilon})$
and
\(v_{\varepsilon} |\nabla u_{\varepsilon}|^2 \in L^1((t, t+\tau); L^1(\Omega)),
\)
by applying Lemma \ref{lemma-3.5} together with the differential inequality for the 
above energy-like functional.

\begin{lem}\label{lemma-3.8}
Assume that \eqref{assIniVal} holds. Then there exists constant $C>0$, independent of $t$ and $\varepsilon$, such that 
\begin{align}\label{0926-14}
\int_{\Omega} \frac{|\nabla v_{\varepsilon}|^4}{v_{\varepsilon}^3}  \leqslant C  \quad \text { for all } t \in\left(0, {T}_{\max,\varepsilon } \right) \text { and } \varepsilon \in(0,1)
\end{align}
and
\begin{align}\label{0927-14}
\int_t^{t+\tau} \int_{\Omega} v_{\varepsilon}\left|\nabla u_{\varepsilon}\right|^2  \leqslant C \quad \text { for all } t \in\left(0, \widetilde T_{\max, \varepsilon}\right) \text {  and   } \varepsilon \in(0,1)
\end{align}
as well as
\begin{align}\label{0927-11}
\int_t^{t+\tau} \int_{\Omega}\frac{\left|\nabla v_{\varepsilon}\right|^6}{v_{\varepsilon}^5} \leqslant C  \quad \text { for all } t \in\left(0,  \widetilde{T}_{\max, \varepsilon} \right) \text { and } \varepsilon \in(0,1)
\end{align}
where $\tau$ and $\widetilde{T}_{\max, \varepsilon}$ are defined by \eqref{-3.5ss}.
\end{lem}

\begin{proof}
We use \eqref{0903-2036} to derive the estimates in this lemma. An application of \eqref{eq-6.4} with $p=1$ and $\eta =\frac{1}{8}$ provides $c_1=c_1(\Omega)>0$ such that
\begin{align*}
 4b  \int_\Omega u_{\varepsilon}^{2} v_{\varepsilon} |\nabla v_{\varepsilon}|^2
\leqslant & \frac{b}{2} \int_\Omega  v_{\varepsilon}|\nabla u_{\varepsilon}|^2 + \left\{c_1 b \left\|v_{\varepsilon}\right\|_{L^{\infty}(\Omega)}+8 c_1 b  \left\|v_{\varepsilon} \right\|^3_{L^{\infty}(\Omega)}\right\} \cdot \int_{\Omega} u_{\varepsilon}^{2} v_{\varepsilon} \cdot \int_\Omega \frac{|\nabla v_{\varepsilon}|^4}{v_{\varepsilon}^3} \nonumber\\
&+c_1 b  \left\|v_{\varepsilon} \right\|^2_{L^{\infty}(\Omega)} \cdot\left\{\int_\Omega u_{\varepsilon}\right\}^{3} \cdot \int_\Omega \frac{|\nabla v_{\varepsilon}|^4}{v_{\varepsilon}^3}+c_1 b  \left\|v_{\varepsilon} \right\|^2_{L^{\infty}(\Omega)}\cdot \int_{\Omega} u_{\varepsilon} v_{\varepsilon}.
\end{align*}
Substituting this into \eqref{0903-2036} and using $\xi \ln \xi +\frac{1}{e}\geqslant 0$ for all $\xi>0$, thanks to \eqref{-2.9} and \eqref{-3.4}, we obtain 
\begin{align}\label{-3.33}
\dt\left\{4 b \int_{\Omega} u_{\varepsilon} \ln u_{\varepsilon} +\int_{\Omega} \frac{\left|\nabla v_{\varepsilon}\right|^4}{v_{\varepsilon}^3}\right\} & + \frac{b}{2} \int_{\Omega} v_{\varepsilon}\left|\nabla u_{\varepsilon}\right|^2 \nonumber\\
& + \int_{\Omega} u_{\varepsilon} v_{\varepsilon}^{-3}\left|\nabla v_{\varepsilon}\right|^4 + \int_{\Omega} v_{\varepsilon}^{-1}\left|\nabla v_{\varepsilon}\right|^2\left|D^2 \ln v_{\varepsilon}\right|^2\nonumber\\
\leqslant & B \int_{\Omega} u_{\varepsilon}^{2} \cdot \int_\Omega \frac{|\nabla v_{\varepsilon}|^4}{v_{\varepsilon}^3} +c_1 b  m^3  \left\|v_{0} \right\|^2_{L^{\infty}(\Omega)}\cdot \int_\Omega \frac{|\nabla v_{\varepsilon}|^4}{v_{\varepsilon}^3}\nonumber\\
& +4  b m  + c_1 b \left\|v_{0} \right\|^2_{L^{\infty}(\Omega)} \cdot \int_{\Omega} u_{\varepsilon} v_{\varepsilon}+ c \int_{\Omega} v_{\varepsilon}\nonumber\\
\leqslant & B \int_{\Omega} u_{\varepsilon}^{2}  \cdot \left(4 b \int_{\Omega} u_{\varepsilon} \ln u_{\varepsilon} +\int_{\Omega} \frac{\left|\nabla v_{\varepsilon}\right|^4}{v_{\varepsilon}^3}+\frac{4b|\Omega|}{e}\right) \nonumber \\
& +c_1 b  m^3  \left\|v_{0} \right\|^2_{L^{\infty}(\Omega)}\cdot \int_\Omega \frac{|\nabla v_{\varepsilon}|^4}{v_{\varepsilon}^3} +4  b m \nonumber\\
& + c_1 b \left\|v_{0} \right\|^2_{L^{\infty}(\Omega)} \cdot \int_{\Omega} u_{\varepsilon} v_{\varepsilon}+c \int_{\Omega} v_{\varepsilon} \nonumber\\
= & B \int_{\Omega} u_{\varepsilon}^{2}   \cdot \left(4 b \int_{\Omega} u_{\varepsilon} \ln u_{\varepsilon} +\int_{\Omega} \frac{\left|\nabla v_{\varepsilon}\right|^4}{v_{\varepsilon}^3}\right) \nonumber \\
& +c_1 b  m^3  \left\|v_{0} \right\|^2_{L^{\infty}(\Omega)}\cdot \int_\Omega \frac{|\nabla v_{\varepsilon}|^4}{v_{\varepsilon}^3}+ 4bm \nonumber\\
&  + c_1 b \left\|v_{0} \right\|^2_{L^{\infty}(\Omega)} \cdot \int_{\Omega} u_{\varepsilon} v_{\varepsilon}  + \frac{4  b B|\Omega|}{e} \cdot \int_{\Omega} u^2_{\varepsilon} \nonumber\\
& + c \int_{\Omega} v_{\varepsilon}
\end{align}
for all $t \in(0, T_{\max, \varepsilon})$ and $\varepsilon \in(0,1)$, where  $B=c_1 b \left\|v_{0} \right\|^2_{L^{\infty}(\Omega)}+8 c_1 b \left\|v_{0} \right\|^4_{L^{\infty}(\Omega)}$. Setting
\begin{align*}
F_{\varepsilon}(t):= 4b \int_{\Omega} u_{\varepsilon} \ln u_{\varepsilon}+\int_{\Omega} \frac{\left|\nabla v_{\varepsilon}\right|^4}{v_{\varepsilon}^3}
\end{align*}
and
\begin{align*}
Z_{\varepsilon}(t):= B\int_{\Omega} u_{\varepsilon}^{2} 
\end{align*}
as well as
\begin{align*}
M_{\varepsilon}(t) := & c_1 b  m^3 \left\|v_{0} \right\|^2_{L^{\infty}(\Omega)}   \cdot \int_\Omega \frac{|\nabla v_{\varepsilon}|^4}{v_{\varepsilon}^3} + 4bm \\
& + c_1 b \left\|v_{0} \right\|^2_{L^{\infty}(\Omega)} \cdot \int_{\Omega} u_{\varepsilon} v_{\varepsilon} +  \frac{4bB|\Omega|}{e} \cdot \int_{\Omega} u^2_{\varepsilon}+ c \int_{\Omega} v_{\varepsilon},  
\end{align*}
We rewrite \eqref{-3.33} as follows
\begin{align}\label{-3.34}
F_{\varepsilon}^{\prime}(t) \leqslant  Z_{\varepsilon}(t) F_{\varepsilon}(t)+M_{\varepsilon}(t) \quad \text { for all } t \in(0, T_{\max, \varepsilon}) \text { and } \varepsilon \in(0,1). 
\end{align}
By Lemmas \ref{lem-1st-est2} and \ref{lem-0927-1}, we have 
\begin{align*}
\int_t^{t+\tau} Z_{\varepsilon}(\sigma) \d \sigma \leqslant  c_2 \quad \text { for all } t \in\left(0, \widetilde T_{\max, \varepsilon}\right) \text {  and   } \varepsilon \in(0,1)
\end{align*}
and
\begin{align*}
\int_t^{t+\tau} F_{\varepsilon}(\sigma) \d \sigma \leqslant  c_2 \quad \text { for all } t \in\left(0, \widetilde T_{\max, \varepsilon}\right) \text {  and   } \varepsilon \in(0,1).
\end{align*}
Due to \eqref{-3.4ss}, \eqref{-2.10}, \eqref{-3.5} and \eqref{0926-12}, we get
\begin{align}\label{1119-11}
\int_t^{t+\tau} M_{\varepsilon}(s) \d s \leqslant  c_3 \quad \text { for all } t \in\left(0, \widetilde T_{\max, \varepsilon}\right) \text {  and   } \varepsilon \in(0,1).
\end{align}
Applying Lemma \ref{lem-0926-2} to \eqref{-3.34} implies that 
\begin{align}\label{0908-1147}
F_{\varepsilon}(t):=4b \int_{\Omega} u_{\varepsilon} \ln u_{\varepsilon}+\int_{\Omega} \frac{\left|\nabla v_{\varepsilon}\right|^4}{v_{\varepsilon}^3} & \leqslant \left(\frac{a_2}{\tau}+c_3\right) e^{c_2}
\end{align}
for all $t \in(0, T_{\max, \varepsilon})$ and $\varepsilon \in(0,1)$. Therefore, this  proves \eqref{0926-14}.
By a direct integration in \eqref{-3.33}, thanks to \eqref{-3.4ss}, \eqref{-2.10}, \eqref{-3.5}, \eqref{0926-12}, \eqref{-3.21}  and \eqref{0908-1147}, we see that
\begin{align*}
\int_t^{t+\tau} \int_{\Omega} v_{\varepsilon}\left|\nabla u_{\varepsilon}\right|^2  \leqslant C \quad \text { for all } t \in\left(0, \widetilde T_{\max, \varepsilon}\right) \text {  and   } \varepsilon \in(0,1).
\end{align*}
and
\begin{align*}
\int_t^{t+\tau} \int_{\Omega} \frac{\left|\nabla v_{\varepsilon}\right|^6}{v_{\varepsilon}^5} \leqslant C \quad \text { for all } t \in\left(0, \widetilde T_{\max, \varepsilon}\right) \text {  and   } \varepsilon \in(0,1).
\end{align*}
We complete the proof of \eqref{0927-14} and \eqref{0927-11}.
\end{proof}

In light of Lemmas \ref{lem-0926-3} and \ref{lemma-3.8}, we obtain the following result.

\begin{lem}\label{lemma-3.9x}
Assume that \eqref{assIniVal} holds. Then for all $p > 1$, there exists constant $C(p)>0$, independent of $t$ and $\varepsilon$, such that 
\begin{align}\label{-3.36}
\int_{\Omega} u_{\varepsilon}^p(t) \leqslant C(p) \quad \text { for all } t \in(0, T_{\max, \varepsilon}) \text { and } \varepsilon \in(0,1).
\end{align}
\end{lem}

\begin{proof}
We employ \eqref{0926-14} to fix $c_1 > 0$ such that
\begin{align}\label{0904-923}
\int_{\Omega} \frac{|\nabla v_{\varepsilon}|^4}{v_{\varepsilon}^3} \leqslant c_1.
\end{align}
Plugging \eqref{-2.9}, \eqref{-3.4} and
\eqref{0904-923} into \eqref{eq-3.1} and applying Young's inequality,  we obtain 
\begin{align*}
& \dt \int_{\Omega} u_{\varepsilon}^p 
+  \frac{p(p-1)}{4} \int_{\Omega} u_{\varepsilon}^{p-1} v_{\varepsilon} \left|\nabla 
u_{\varepsilon}\right|^2 + \frac{p}{2} \int_{\Omega} u_{\varepsilon}^{p+1} \nonumber\\
\leqslant &   c_1 A  \left\{
\int_{\Omega} v_{\varepsilon} |\nabla u_{\varepsilon}|^2
+\int_{\Omega} \frac{u_{\varepsilon}}{v_{\varepsilon}}|\nabla v_{\varepsilon}|^2
  +\int_{\Omega} u_{\varepsilon} v_{\varepsilon}  \right\} \cdot \int_{\Omega} u_{\varepsilon}^p  \nonumber\\ 
& +c_1  A \int_{\Omega} v_{\varepsilon} |\nabla u_{\varepsilon}|^2 + c_1 A m^{2 p+1}  + A m \|v_0\|_{L^{\infty}(\Omega)} + p  2^p |\Omega|
\end{align*}
for all $t \in(0, T_{\max, \varepsilon})$ and $\varepsilon \in(0,1)$.  By adding $\delta \int_{\Omega} u_{\varepsilon}^p$ to both sides of the above equation, and using estimate $\delta  \int_{\Omega} u_{\varepsilon}^p \leqslant \frac{p}{2} \int_{\Omega} u_{\varepsilon}^{p+1} + \delta |\Omega| \left({\frac{2 \delta}{p}}\right)^{p}$, we have
\begin{align}\label{-3.43}
& \dt \int_{\Omega} u_{\varepsilon}^p 
+  \frac{p(p-1)}{4} \int_{\Omega} u_{\varepsilon}^{p-1} v_{\varepsilon} \left|\nabla 
u_{\varepsilon}\right|^2  + \delta \int_{\Omega} u_{\varepsilon}^p \nonumber\\
\leqslant &   c_1 A  \cdot \left\{
\int_{\Omega} v_{\varepsilon} |\nabla u_{\varepsilon}|^2
+\int_{\Omega} \frac{u_{\varepsilon}}{v_{\varepsilon}}|\nabla v_{\varepsilon}|^2
  +\int_{\Omega} u_{\varepsilon} v_{\varepsilon}  \right\} \cdot \int_{\Omega} u_{\varepsilon}^p  \nonumber\\ 
& +c_1  A \cdot \int_{\Omega} v_{\varepsilon} |\nabla u_{\varepsilon}|^2 +A m \|v_0\|_{L^{\infty}(\Omega)} + c_1 A m^{2 p+1} \nonumber\\ 
& + 2^\frac{1}{p} p |\Omega|+ \delta  |\Omega| \left({\frac{2 \delta}{p}}\right)^{p}
\end{align}
where $\delta >0$ is a constant independent of $t$ and $\varepsilon$. For each $\varepsilon \in(0,1)$, we set
\begin{align*}
Y_{\varepsilon}(t) := \int_{\Omega} u_{\varepsilon}^p
\end{align*}
and
\begin{align*}
H_{\varepsilon}(t):= c_1  A \cdot \left\{\int_{\Omega} v_{\varepsilon} |\nabla u_{\varepsilon}|^2+\int_{\Omega} \frac{u_{\varepsilon}}{v_{\varepsilon}}|\nabla v_{\varepsilon}|^2+ \int_{\Omega} u_{\varepsilon} v_{\varepsilon} \right\}
\end{align*}
as well as
\begin{align*}
G_{\varepsilon}(t):= c_1  A \cdot \int_{\Omega} v_{\varepsilon} |\nabla u_{\varepsilon}|^2 +A m \|v_0\|_{L^{\infty}(\Omega)} + c_1 A m^{2 p+1} + 2^\frac{1}{p} p |\Omega|+ \delta  |\Omega| \{\frac{2 \delta}{p}\}^{p}.
\end{align*}
Thus \eqref{-3.43} becomes 
\begin{align}\label{-3.44}
Y_{\varepsilon}^{\prime}(t) + \delta  Y_{\varepsilon}(t) \leqslant  H_{\varepsilon}(t) Y_{\varepsilon}(t)+ G_{\varepsilon}(t) \quad \text { for all } t \in(0, T_{\max, \varepsilon}) \text { and } \varepsilon \in(0,1). 
\end{align}
According to \eqref{-2.10}, \eqref{0927-14} and \eqref{0926-13}, we obtain
\begin{align*}
\int_t^{t+\tau} H_{\varepsilon}(\sigma) \d \sigma \leqslant c_3(p) \quad \text { for all } t \in\left(0, \widetilde T_{\max, \varepsilon}\right) \text {  and   } \varepsilon \in(0,1).
\end{align*}
Due to \eqref{0927-14}, we have
\begin{align*}
\int_t^{t+\tau} G_{\varepsilon}(s) \d s \leqslant c_4(p) \quad \text { for all } t \in\left(0, \widetilde T_{\max, \varepsilon}\right) \text {  and   } \varepsilon \in(0,1),
\end{align*}
where $\tau$ and $\widetilde{T}_{\max, \varepsilon}$ are defined by \eqref{-3.5ss}. Now, let $\delta=\frac{1+c_3(p)}{\tau}$. Then the application of Lemma \ref{lem-0926-3} to \eqref{-3.44}
with $b_1=c_3(p)$, $c_1=c_4(p)$ and $\varrho=1$ yields
$$
\int_{\Omega} u_{\varepsilon}^p \leqslant e^{b_1} \cdot \int_{\Omega} (u_{0}+1)^p + \frac{c_1 e^{2 b_1}}{1-e^{-1}}+c_1 e^{b_1} \quad \text { for all } t \in\left(0, T_{\max, \varepsilon}\right) \text {  and   } \varepsilon \in(0,1).
$$
We complete the proof \eqref{-3.36}.
\end{proof}

\section{Uniform $L^\infty \times W^{1,\infty}$ Boundedness of $(u_\varepsilon, v_\varepsilon)$}\label{section5}
Based on standard heat semigroup estimates, we can obtain $L^\infty$ bounds for $\nabla v_{\varepsilon}$.

\begin{lem}\label{lemma-4.1}
Assume that \eqref{assIniVal} holds. Then there exists constant $C>0$, independent of $t$ and $\varepsilon$, such that 
\begin{align}\label{-4.1}
\left\|v_{\varepsilon}(t)\right\|_{W^{1, \infty}(\Omega)} \leqslant C\quad \text { for all } t \in(0, T_{\max, \varepsilon}) \text { and } \varepsilon \in(0,1).
\end{align}
\end{lem}
\begin{proof}
According to the well-known smoothing properties of the Neumann heat semigroup $\left(e^{t \Delta}\right)_{t \geqslant 0}$ on $\Omega$ (cf. \cite{2010-JDE-Winkler}), fixing any $p > 1$ we can find $c_1=c_1(\Omega)>0$ such that 
\begin{align}\label{-4.2}
&\left\|\nabla v_{\varepsilon}(t)\right\|_{L^{\infty}(\Omega)} \nonumber\\
= &\left\|\nabla e^{t(\Delta-1)} v_0-\int_0^t \nabla e^{(t-s)(\Delta-1)}\left\{u_{\varepsilon}(s) v_{\varepsilon}(s)-v_{\varepsilon}(s)\right\} \d s\right\|_{L^{\infty}(\Omega)} \nonumber\\
\leqslant &  c_1\left\|v_0\right\|_{W^{1, \infty}(\Omega)}\nonumber\\
& +  c_1 \int_0^t\left(1+(t-s)^{-\frac{1}{2}-\frac{1}{p}}\right) e^{-(t-s)}\left\|u_{\varepsilon}(s) v_{\varepsilon}(s)-v_{\varepsilon}(s)\right\|_{L^p(\Omega)} \d s
\end{align}
for all $t \in\left(0, T_{\max, \varepsilon}\right)$ and $\varepsilon \in(0,1)$. In view of \eqref{-2.9} and \eqref{-3.36}, we have
\begin{align*}
\left\|u_{\varepsilon}(s) v_{\varepsilon}(s)-v_{\varepsilon}(s)\right\|_{L^p(\Omega)} 
\leqslant & \left\|u_{\varepsilon}(s)\right\|_{L^p(\Omega)}\left\|v_{\varepsilon}(s)\right\|_{L^{\infty}(\Omega)}+|\Omega|^{\frac{1}{p}}\left\|v_{\varepsilon}(s)\right\|_{L^{\infty}(\Omega)} \\
\leqslant & c_2\left\|v_0\right\|_{L^{\infty}(\Omega)}+|\Omega|^{\frac{1}{p}}\left\|v_0\right\|_{L^{\infty}(\Omega)}
\end{align*}
where $c_2=\sup _{\varepsilon \in(0,1)} \sup _{t \in\left(0, T_{\max, \varepsilon}\right)}\left\|u_{\varepsilon}(s)\right\|_{L^p(\Omega)}<+\infty$. This together with \eqref{-4.2} yields \eqref{-4.1}.
\end{proof}

As a further prerequisite, we introduce the following condition that allows us to control the time evolution of certain singularly weighted integrals involving the gradient of the signal.

\begin{lem}\label{lem-1001-01}
Let $\Omega \subset \mathbb{R}^2$ be a bounded domain with smooth boundary and $q \geqslant 2$. Then there exists $\gamma(q)>0$ such that for all $\varphi \in C^1(\bar{\Omega})$ and $\psi \in C^1(\bar{\Omega})$ fulfilling $\varphi>0$ and $\psi>0$ in $\bar{\Omega}$, we have
\begin{align*}
\frac{d}{d t} \int_{\Omega} \psi^{-q+1}\left|\nabla \psi\right|^q+\gamma(q) \int_{\Omega} \psi^{-q-1}\left|\nabla \psi\right|^{q+2} \leqslant \frac{1}{\gamma(q)} \left(\int_{\Omega} \varphi^{\frac{q+2}{2}} \psi + \int_{\Omega} \psi \right).
\end{align*}
\end{lem}
\begin{proof}
It follows from Lemma 2.3 of \cite{2024-JDE-Winkler} that
\begin{align*}
\frac{d}{d t} \int_{\Omega} \psi^{-q+1}\left|\nabla \psi\right|^q & + c_1 \int_{\Omega} \psi^{-q+1}\left|\nabla \psi\right|^{q-2}\left|D^2 \psi\right|^2 + c_2 \int_{\Omega} \psi^{-q-1}\left|\nabla\psi\right|^{q+2} \nonumber\\
& \leqslant -q \int_{\Omega} \psi^{-q+1}\left|\nabla \psi\right|^{q-2} \nabla \psi \cdot \nabla\left(\varphi \psi\right) + (q-1) \int_{\Omega} \varphi \psi^{-q+1}\left|\nabla \psi \right|^q \nonumber\\
& \quad + \frac{q}{2} \int_{\partial \Omega} \psi^{-q+1}\left|\nabla \psi\right|^{q-2} \cdot \frac{\partial\left|\nabla \psi \right|^2}{\partial \nu} \nonumber\\
& \leqslant \frac{c_1}{2}\int_{\Omega} \psi^{-q+1}\left|\nabla \psi\right|^{q-2}\left|D^2 \psi\right|^2 + \frac{c_2}{2} \int_{\Omega} \psi^{-q-1}\left|\nabla \psi\right|^{q+2} \nonumber\\
& \quad + \left\{\left(\frac{2}{c_2}\right)^{\frac{q-2}{q+2}} c_3\right\}^{\frac{q+2}{4}} \int_{\Omega}\varphi^{\frac{q+2}{2}} \psi + \frac{q}{2} \int_{\partial \Omega} \psi^{-q+1}\left|\nabla \psi\right|^{q-2} \cdot \frac{\partial\left|\nabla \psi \right|^2}{\partial \nu}
\end{align*}
where $c_1 :=\frac{q}{2(q+1+\sqrt{2})^2}$, $c_2 :=\frac{q}{2(q+\sqrt{2})^2}$ and $c_3 :=\frac{q^2(q-2+\sqrt{2})^2}{4 c_1}$. Lemma 3.5 of \cite{2022-DCDSSB-Winkler} gives that
\begin{align*}
\frac{q}{2}\int_{\partial \Omega} \psi^{-q+1}|\nabla \psi|^{q-2} \cdot \frac{\partial|\nabla \psi|^2}{\partial \nu} \leqslant \frac{c_1}{2} \int_{\Omega} \psi^{-q+1}|\nabla \psi|^{q-2}\left|D^2 \psi\right|^2 + \frac{c_2}{4} \int_{\Omega} \psi^{-q-1}|\nabla \psi|^{q+2} + C \int_{\Omega} \psi
\end{align*}
where $C$ is a positive constant depending only on $c_1$, $c_2$ and $q$. A combination of the above two inequalities yields that
\begin{align*}
\frac{d}{d t} \int_{\Omega} \psi^{-q+1}\left|\nabla \psi\right|^q + \frac{c_2}{4} \int_{\Omega} \psi^{-q-1}\left|\nabla\psi\right|^{q+2} \leqslant K \int_{\Omega}\varphi^{\frac{q+2}{2}} \psi + C \int_{\Omega} \psi.
\end{align*}
where $K := \left(\frac{2}{c_2}\right)^{\frac{q-2}{4}} c_3^{\frac{q+2}{4}}$.
Thus, we complete the proof by letting $\gamma(q):= \min\left\{ \frac{c_2}{4}, \frac{1}{\max\{K, C\}} \right\}$.
\end{proof}

Guided by the estimates established in the preceding lemmas, we next derive the following lemma that lays the groundwork for obtaining global-in-time $L^\infty$ bounds for $u_{\varepsilon}$.

\begin{lem}\label{1002-lemma-11}
Assume that \eqref{assIniVal} holds. Then there exists constant $C>0$, independent of $t$ and $\varepsilon$, such that 
\begin{align}\label{1001-101}
\int_{\Omega}\frac{\left|\nabla v_{\varepsilon}\right|^6}{v_{\varepsilon}^5} \leqslant C  \quad \text { for all } t \in\left(0,  {T}_{\max, \varepsilon} \right) \text { and } \varepsilon \in(0,1).
\end{align}
\end{lem}

\begin{proof}
Lemma \ref{lem-1001-01} establishes the differential inequality
\begin{align}\label{1001-12}
\frac{d}{d t} \int_{\Omega} \frac{\left|\nabla v_{\varepsilon}\right|^6}{v_{\varepsilon}^5} +\gamma(4) \int_{\Omega} \frac{\left|\nabla v_{\varepsilon}\right|^7}{v_{\varepsilon}^8} \leqslant \frac{1}{\gamma(4)} \left(\int_{\Omega} u_{\varepsilon}^{4} v_{\varepsilon} + \int_{\Omega} v_{\varepsilon} \right).
\end{align}
for all $t \in(0, T_{\max, \varepsilon})$ and $\varepsilon \in(0,1)$. Since Lemma \ref{lemma-3.9x} provides that $\int_{\Omega} u_{\varepsilon}^3(t) \leqslant c_1$, and by applying \eqref{eq-6.1}, we have 
\begin{align*}
\int_{\Omega} u_{\varepsilon}^{4} v_{\varepsilon} \leqslant c_1 c_2 \left\{\int_{\Omega} v_{\varepsilon} |\nabla u_{\varepsilon}|^2
  +\int_{\Omega} \frac{u_{\varepsilon}}{v_{\varepsilon}}|\nabla v_{\varepsilon}|^2\
    +\int_{\Omega} u_{\varepsilon} v_{\varepsilon}  \right\} + c_2 \int_{\Omega} v_{\varepsilon} |\nabla u_{\varepsilon}|^2.
\end{align*}
for all $t \in(0, T_{\max, \varepsilon})$ and $\varepsilon \in(0,1)$. Furthermore, from \eqref{-2.10}, \eqref{0926-13} and \eqref{0927-14}, we have 
\[
\int_t^{t+\tau} \int_{\Omega} v_{\varepsilon} |\nabla u_{\varepsilon}|^2 +\int_t^{t+\tau} \int_{\Omega} \frac{u_{\varepsilon}}{v_{\varepsilon}}|\nabla v_{\varepsilon}|^2+\int_t^{t+\tau} \int_{\Omega} u_{\varepsilon} v_{\varepsilon} \leqslant c_3\quad \text { for all } t \in\left(0,  \widetilde{T}_{\max, \varepsilon } \right) \text { and } \varepsilon \in(0,1).
\]
This implies that
\begin{align}\label{1001-13}
\int_t^{t+\tau} \int_{\Omega} u_{\varepsilon}^{4} v_{\varepsilon} \leqslant (c_1 +1) c_2 c_3\quad \text { for all } t \in\left(0,  \widetilde{T}_{\max, \varepsilon } \right) \text { and } \varepsilon \in(0,1).
\end{align}
where $\tau$ and $\widetilde{T}_{\max }$ are defined by \eqref{-3.5ss}.
To proceed, we add the term $\int_{\Omega}\frac{\left|\nabla v_{\varepsilon}\right|^6}{v_{\varepsilon}^5}$ to both sides of inequality \eqref{1001-12}, yielding
\begin{align}\label{1001-14}
\frac{d}{d t} \int_{\Omega} \frac{\left|\nabla v_{\varepsilon}\right|^6}{v_{\varepsilon}^5} + \int_{\Omega}\frac{\left|\nabla v_{\varepsilon}\right|^6}{v_{\varepsilon}^5} \leqslant \int_{\Omega}\frac{\left|\nabla v_{\varepsilon}\right|^6}{v_{\varepsilon}^5} + \frac{1}{\gamma(4)} \left(\int_{\Omega} u_{\varepsilon}^{4} v_{\varepsilon} + \int_{\Omega} v_{\varepsilon} \right).
\end{align}
for all $t \in(0, T_{\max, \varepsilon})$ and $\varepsilon \in(0,1)$. Combining the above inequality with \eqref{-2.9}, \eqref{0927-11}, \eqref{1001-13} and Lemma \ref{lem-0926-1}, we complete the proof of \eqref{1001-101}.
\end{proof}

We can now proceed to assert the boundedness of $u_{\varepsilon}$.
\begin{lem}\label{lemma-4.4}
Assume that \eqref{assIniVal} holds. Then there exists constant $C>0$, independent of $t$ and $\varepsilon$, such that  
\begin{align}\label{619-1648}
\left\|u_{\varepsilon}(t)\right\|_{L^{\infty}(\Omega)} \leqslant C \quad \text { for all } t \in(0, T_{\max, \varepsilon}) \text { and } \varepsilon \in(0,1).
\end{align}
\end{lem}
\begin{proof}
Based on Lemma \ref{lemma-4.1}, we can find $c_1>0$ such that
\begin{align}\label{511-1312}
\left|\nabla v_{\varepsilon}(x, t)\right| \leqslant c_1 \quad \text { for all } x \in \Omega, t \in\left(0, T_{\max, \varepsilon}\right) \text { and } \varepsilon \in(0,1).
\end{align}
Let $p_0$ be an arbitrarily fixed number greater than $1$. Then for integers $k \geqslant 1$ we set $p_k=2^k p_0$, and let
\begin{align}\label{511-1223-1}
N_{k, \varepsilon}(T)=1+\sup _{t \in(0, T)} \int_{\Omega} u_{\varepsilon}^{p_k}(t)< \infty, \quad T \in\left(0, T_{\max, \varepsilon}\right), k \in\{0,1,2, \ldots\}, \varepsilon \in(0,1).
\end{align}
From \eqref{-3.36}, we infer that there exists a constant $c_2>0$ such that
\begin{align}\label{511-1223}
N_{0, \varepsilon}(T) \leqslant c_2 \quad \text { for all } T \in\left(0, T_{\max, \varepsilon}\right) \text { and } \varepsilon \in(0,1). 
\end{align}
To estimate $N_{k, \varepsilon}(T)$ for $k \geqslant 1, T \in\left(0, T_{\max, \varepsilon}\right)$ and $\varepsilon \in(0,1)$, using the first equation of the system \eqref{sys-regul}, according to \eqref{511-1312} and Young's inequality, we see that 
\begin{align}\label{511-1339}
\dt \int_{\Omega} u_{\varepsilon}^{p_k}= & -p_k\left(p_k-1\right) \int_{\Omega} u_{\varepsilon}^{p_k-1} v_{\varepsilon}\left|\nabla u_{\varepsilon}\right|^2+ p_k\left(p_k-1\right) \int_{\Omega} u_{\varepsilon}^{p_k} v_{\varepsilon} \nabla u_{\varepsilon} \cdot \nabla v_{\varepsilon}\nonumber \\
& + p_k \int_{\Omega} u_{\varepsilon}^{p_k} - p_k \int_{\Omega} u_{\varepsilon}^{p_k+1} \nonumber \\
\leqslant & -\frac{ p_k\left(p_k-1\right)}{2} \int_{\Omega} u_{\varepsilon}^{p_k-1} v_{\varepsilon}\left|\nabla u_{\varepsilon}\right|^2+\frac{ p_k\left(p_k-1\right)}{2} \int_{\Omega} u_{\varepsilon}^{p_k+1} v_{\varepsilon}\left|\nabla v_{\varepsilon}\right|^2 \nonumber\\
& +p_k  \int_{\Omega} u_{\varepsilon}^{p_k} - p_k  \int_{\Omega} u_{\varepsilon}^{p_k+1}\nonumber \\
\leqslant & -\frac{p_k\left(p_k-1\right)}{2} \int_{\Omega} u_{\varepsilon}^{p_k-1} v_{\varepsilon}\left|\nabla u_{\varepsilon}\right|^2+ p_k\left(p_k-1\right) c^2_1 \int_{\Omega} u_{\varepsilon}^{p_k+1} v_{\varepsilon} \nonumber \\
& - \int_{\Omega}   u_{\varepsilon}^{p_k}+ \left(\frac{p_k+1}{p_k}\right)^{p_k}|\Omega|
\end{align}
Since $p_k\left(p_k-1\right) \leqslant p_k^2$, $p_k \leqslant p_k^2$, $\frac{p_k\left(p_k-1\right)}{2} \geqslant \frac{p_k^2}{4}$ and $\left(\frac{p_k+1}{p_k}\right)^{p_k} \leqslant 2^{p_k}$, for all $t \in\left(0, T_{\max, \varepsilon}\right)$, we from \eqref{511-1339} infer that
\begin{align}\label{511-1338}
\dt \int_{\Omega} u_{\varepsilon}^{p_k}+\int_{\Omega} u_{\varepsilon}^{p_k}+\frac{p_k^2}{4} \int_{\Omega} u_{\varepsilon}^{p_k-1} v_{\varepsilon}\left|\nabla u_{\varepsilon}\right|^2 \leqslant 2^{p_k}|\Omega| +c_3 p_k^2 \int_{\Omega} u_{\varepsilon}^{p_k+1} v_{\varepsilon} 
\end{align}
where $c_3=c^2_1$. According to Lemma \ref{1002-lemma-11}, there exists $c_4>0$ such that
\begin{align}\label{1001-11}
\int_{\Omega}\frac{\left|\nabla v_{\varepsilon}\right|^6}{v_{\varepsilon}^5} \leqslant c_4  \quad \text { for all } t \in\left(0,  {T}_{\max, \varepsilon} \right) \text { and } \varepsilon \in(0,1).
\end{align}
By choosing $p_{\star}=2 p_0>2$ in \cite[Lemma 6.2]{2024-JDE-Winkler} and invoking \eqref{-2.9} and \eqref{-3.4}, we conclude that
\begin{align*}
\int_{\Omega} u_{\varepsilon}^{p_k+1} v_{\varepsilon} \leqslant & \frac{1}{4 c_3} \int_{\Omega} u_{\varepsilon}^{p_k-1} v_{\varepsilon}\left|\nabla u_{\varepsilon}\right|^2+\frac{c_4}{4 c_3} \cdot\left\{\int_{\Omega} u_{\varepsilon}^{\frac{p_k}{2}}\right\}^{\frac{2\left(p_k+1\right)}{p_k}} \\
& +K \cdot m \cdot\left(4 c_3\right)^\kappa \cdot p_k^{2 \kappa} \cdot \|v_0\|_{L^{\infty}(\Omega)} \cdot\left\{\int_{\Omega} u_{\varepsilon}^{\frac{p_k}{2}}\right\}^2  \quad \text { for all } t \in\left(0, T_{\max, \varepsilon}\right), 
\end{align*}
where $\kappa=\kappa\left(p_{\star}\right)>0$ and $K=K\left(p_{\star}\right)>0$. According to \eqref{511-1223-1}, for all $T \in\left(0, T_{\max, \varepsilon}\right)$ we can estimate
\begin{align*}
\int_{\Omega} u_{\varepsilon}^{\frac{p_k}{2}}=\int_{\Omega} u_{\varepsilon}^{p_{k-1}} \leqslant N_{k-1, \varepsilon}(T) \quad \text { for all } t \in(0, T).
\end{align*}
We then obtain
\begin{align*}
c_3 p_k^2 \int_{\Omega} u_{\varepsilon}^{p_k+1} v_{\varepsilon} \leqslant & \frac{p_k^2}{4} \int_{\Omega} u_{\varepsilon}^{p_k-1} v_{\varepsilon}\left|\nabla u_{\varepsilon}\right|^2+\frac{c_4  p_k^2}{4} N_{k-1, \varepsilon}^{\frac{2\left(p_k+1\right)}{p_k}}(T)\\
& +m 4^\kappa c_3^{\kappa+1}\|v_0\|_{L^{\infty}(\Omega)} K p_k^{2 \kappa+2} N_{k-1, \varepsilon}^2(T) .
\end{align*}
Since $p_k^2 \leqslant p_k^{2 \kappa+2}$ and $1 \leqslant N_{k-1, \varepsilon}^2(T) \leqslant N_{k-1, \varepsilon}^{\frac{2\left(p_k+1\right)}{p_k}}(T)$, from \eqref{511-1338} we deduce that with $c_5 =\max \left\{\frac{c_4}{4}, m 4^\kappa c_3^{\kappa+1}\|v_0\|_{L^{\infty}(\Omega)} K,1\right\}$, 
$$
\dt \int_{\Omega} u_{\varepsilon}^{p_k} + \int_{\Omega} u_{\varepsilon}^{p_k}
\leqslant   2^{p_k}|\Omega| + c_5 p_k^{2 \kappa+2} N_{k-1, \varepsilon}^{\frac{2\left(p_k+1\right)}{p_k}}(T) 
$$
for all $t \in(0, T)$, any $T \in\left(0, T_{\max, \varepsilon}\right)$ and each $\varepsilon \in(0,1)$.
An integration of this ODI provides that for all $t \in(0, T)$, $T \in\left(0, T_{\max, \varepsilon}\right)$ and $\varepsilon \in(0,1)$,
\begin{align*}
\int_{\Omega} u_{\varepsilon}^{p_k} & \leqslant \int_{\Omega}\left(u_0+\varepsilon\right)^{p_k}+  2^{p_k}|\Omega|  + c_5 p_k^{2 \kappa+2} N_{k-1, \varepsilon}^{\frac{2\left(p_k+1\right)}{p_k}}(T) \\
& \leqslant|\Omega| \cdot (\left\|u_0+1\right\|_{L^{\infty}(\Omega)}^{p_k}+ 2^{p_k}) + c_5 p_k^{2 \kappa+2} N_{k-1, \varepsilon}^{\frac{2\left(p_k+1\right)}{p_k}}(T).
\end{align*}
In view of \eqref{511-1223-1}, this yields
\begin{align*}
N_{k, \varepsilon}({T}) & \leqslant 1+|\Omega| \cdot (\left\|u_0+1\right\|_{L^{\infty}(\Omega)}+2)^{p_k}+ c_5 p_k^{2 \kappa+2} N_{k-1, \varepsilon}^{\frac{2\left(p_k+1\right)}{p_k}}(T) \\
& \leqslant s^{2^k}+d^k N_{k-1, \varepsilon}^{2+r \cdot 2^{-k}}({T}) \quad \text { for all } {T} \in\left(0, T_{\max, \varepsilon}\right) \text { and } \varepsilon \in(0,1)
\end{align*}
with
\begin{align*}
d=\left(\max \left\{2  c_5 p_0, 1\right\}\right)^{2 \kappa+2}, \quad 
r=\frac{2}{p_0} \text { and } \quad 
s=[1+ \max \left\{1,|\Omega|\right\} \cdot (\left\|u_0+1\right\|_{L^{\infty}(\Omega)}+2)]^{p_0},
\end{align*}
where the following estimate is used 
\begin{align*}
1+|\Omega| \cdot (\left\|u_0+1\right\|_{L^{\infty}(\Omega)}+2)^{p_k} & \leqslant (1+ \max \left\{1,|\Omega|\right\} \cdot (\left\|u_0+1\right\|_{L^{\infty}(\Omega)}+2))^{p_k}\\
& =[(1+ \max \left\{1,|\Omega|\right\} \cdot (\left\|u_0+1\right\|_{L^{\infty}(\Omega)}+2))^{p_0}]^{2^k}.
\end{align*}
Then it follows from \cite[Lemma 6.3]{2024-JDE-Winkler} and \eqref{511-1223} that
\begin{align*}
\left\|u_{\varepsilon}(t)\right\|_{L^{\infty}(\Omega)}^{p_0} & =\liminf _{k \rightarrow \infty}\left\{\int_{\Omega} u_{\varepsilon}^{p_k}(t)\right\}^{\frac{1}{2^k}} \\
& \leqslant \liminf _{k \rightarrow \infty} N_{k, \varepsilon}^{\frac{1}{2^k}}({T}) \\
& \leqslant\left(2 \sqrt{2} d^3 s^{1+\frac{r}{2}} c_2\right)^{e^{\frac{r}{2}}}
\end{align*}
for all $t \in(0, T)$, $T \in\left(0, T_{\max, \varepsilon}\right)$ and each $\varepsilon \in(0,1)$.
Taking $T \nearrow T_{\max, \varepsilon}$, we obtain \eqref{619-1648}.
\end{proof}

By virtue of Lemma \ref{lemma-4.4}, the global existence of $(u_{\varepsilon}, v_{\varepsilon})$ can be ensured for every $\varepsilon \in (0,1)$.

\begin{lem}\label{lemma-4.5}
Assume that \eqref{assIniVal} holds. Then $T_{\max, \varepsilon}=+\infty$ for all $\varepsilon \in(0,1)$.
\end{lem}
\begin{proof}
This immediately follows from Lemma \ref{lemma-4.4} when combined with \eqref{-2.7}.
\end{proof}

The $L^\infty$ bound obtained in Lemma \ref{lemma-4.4} allows us to establish a lower control on $v_{\varepsilon}$ through a comparison argument.

\begin{lem}\label{lemma-4.6}
Let $T>0$ and assume that \eqref{assIniVal} holds. 
Then we have  
\begin{align}\label{-4.8}
v_{\varepsilon}(x, t) \geqslant C(T)  \quad \text { for all } x \in \Omega, ~~t \in(0, T), \text { and  } \varepsilon \in(0,1),
\end{align}
where $C(T)$ is a positive constant depending on $v_0$, but independent of $\varepsilon$.
\end{lem}
\begin{proof}
Since
\begin{align*}
v_{\varepsilon t} \geqslant \Delta v_{\varepsilon}-c_1 v_{\varepsilon} \quad \text { in } \Omega \times(0, \infty) \quad \text { for all } \varepsilon \in(0,1),
\end{align*}
with $c_1=\sup _{\varepsilon \in(0,1)} \sup _{t>0}\left\|u_{\varepsilon}(t)\right\|_{L^{\infty}(\Omega)}$ being finite by Lemma \ref{lemma-4.4}, from a comparison principle we have
\begin{align*}
v_{\varepsilon}(x, t) \geqslant\left\{\inf _{\Omega} v_0\right\} \cdot e^{-c_1 t} \quad \text { for all } x \in \Omega, ~~t>0 \text { and } \varepsilon \in(0,1).
\end{align*}
We complete the proof.
\end{proof}

Since the boundedness of $u_{\varepsilon}$ and $v_{\varepsilon}$ asserted in Lemmas  \ref{lemma-4.1}, \ref{lemma-4.4} and \ref{lemma-4.6}, the H\"older continuity of 
$u_{\varepsilon}$, $v_{\varepsilon}$, and $\nabla v_{\varepsilon}$ follows from standard 
parabolic regularity theory.
\begin{lem}\label{lemma-4.8}
Let $T>0$ and assume that \eqref{assIniVal} holds. Then one can find $\theta_1=\theta(T) \in(0,1)$  such that 
\begin{align}\label{-4.12}
\left\|u_{\varepsilon}\right\|_{C^{\theta_1, \frac{\theta_1}{2}}(\overline{\Omega} \times[0, T])} \leqslant C_1(T) \quad \text { for all } \varepsilon \in(0,1)
\end{align}
and
\begin{align}\label{-4.13}
\left\|v_{\varepsilon}\right\|_{C^{\theta_1, \frac{\theta_1}{2}}(\overline{\Omega} \times[0, T])} \leqslant C_1(T) \quad \text { for all } \varepsilon \in(0,1),
\end{align}
where $C_1(T)$ is a positive constant independent of $\varepsilon$.
Moreover, for each $\tau>0$ and all $T>\tau$ one can also fix $\theta_2=\theta_2(\tau, T) \in(0,1)$ such that 
\begin{align}\label{-4.14}
\left\|v_{\varepsilon}\right\|_{C^{2+\theta_2, 1+\frac{\theta_2}{2}}(\overline{\Omega} \times [\tau, T])} \leqslant C_2(\tau, T) \quad \text { for all } \varepsilon \in(0,1),
\end{align}
where $C_2(\tau,T)>0$ is a positive constant independent of $\varepsilon$.
\end{lem}
\begin{proof}
It follows from Lemmas \ref{lemma-4.1}, \ref{lemma-4.4}, \ref{lemma-4.6} and \eqref{assIniVal} as well as the H\"{o}lder regularity theory for parabolic equations (cf. \cite{1993-JDE-PorzioVespri}) that \eqref{-4.12} and \eqref{-4.13} hold. Combining the standard Schauder estimates (cf. \cite{1968--LadyzenskajaSolonnikovUralceva}) with \eqref{-4.12} and a cut-off argument, we can deduce \eqref{-4.14}.
\end{proof}

With the above preparations in place, we may now apply a standard extraction argument to obtain a pair of limit functions $(u, v)$, which can be shown to constitute a global bounded weak solution to system \eqref{SYS:MAIN}, as stated in Theorem~\ref{thm-1.1}.

\begin{lem}\label{lemma-4.9}
Assume that the initial value $\left(u_0, v_0\right)$ satisfies \eqref{assIniVal}. Then there exist $(\varepsilon_j)_{j \in \mathbb{N}} \subset(0,1)$ as well as functions $u$ and $v$ which satisfy \eqref{solu:property} with $u \geqslant 0$ and $v>0$ in $\bar{\Omega} \times[0, \infty)$, such that
\begin{flalign}
& u_{\varepsilon} \rightarrow u  \quad \text { in } C_{\mathrm{loc}}^0(\overline{\Omega} \times[0, \infty)) \text {, }\label{-4.15}\\
& v_{\varepsilon} \rightarrow v  \quad \text { in } C_{\mathrm{loc}}^0(\overline{\Omega} \times[0, \infty)) \text { and in } C_{\mathrm{loc}}^{2,1}(\overline{\Omega} \times(0, \infty)),\label{-4.16}\\
& \nabla v_{\varepsilon} \stackrel{*}{\rightharpoonup} \nabla v \quad \text { in } L^{\infty}(\Omega \times(0, \infty)),\label{-4.17}
\end{flalign}
as $\varepsilon=\varepsilon_j \searrow 0$, and that $(u, v)$ is a global weak solution of the system (\ref{SYS:MAIN}) as defined in Definition \ref{def-weak-sol}. 
\end{lem}
\begin{proof}
The existence of a subsequence $\{\varepsilon_j\}_{j\in\mathbb{N}}$ and nonnegative limit functions $u$ and $v$ satisfying \eqref{solu:property} and \eqref{-4.15}--\eqref{-4.17} follows from Lemmas~\ref{lemma-4.1} and~\ref{lemma-4.8} via a standard diagonal extraction argument. Furthermore, the nonnegativity of each $u_{\varepsilon}$, together with \eqref{-4.15} and an application of Fatou’s lemma, implies that $u \geqslant 0$ in $\bar \Omega \times[0, \infty)$. In addition, by \eqref{-4.8} and \eqref{-4.16}, we deduce that $v>0$ in $\bar \Omega \times[0, \infty)$.

From \eqref{-2.9}, \eqref{0926-11}, \eqref{-3.33}, \eqref{0908-1147}, \eqref{-3.36} and Lemma~\ref{lemma-4.5}, we infer that
\begin{align*}
\int_{0}^{T} \int_{\Omega} \frac{u_{\varepsilon}}{v_{\varepsilon}} |\nabla v_{\varepsilon}|^{2} \leqslant C
\qquad \text{for all } T>0 \text{ and } \varepsilon \in (0,1)
\end{align*}
and
\begin{align*}
\int_{0}^{T} \int_{\Omega} v_{\varepsilon} |\nabla u_{\varepsilon}|^{2} \leqslant C
\qquad \text{for all } T>0 \text{ and } \varepsilon \in (0,1).
\end{align*}
A direct integration of \eqref{-3.43}, combined with these inequalities and \eqref{-3.36}, then yields a constant $C(p,T)>0$ such that
\begin{align}\label{615-1428}
\int_{0}^{T} \int_{\Omega} u_{\varepsilon}^{\,p-1} v_{\varepsilon} |\nabla u_{\varepsilon}|^{2} \leqslant C
\qquad \text{for all } T>0 \text{ and } \varepsilon \in (0,1).
\end{align}
Fixing $\beta \in (1,2)$ and $\gamma \in (1,3)$, and applying Young's inequality, we obtain
\begin{align*}
\int_{0}^{T} \int_{\Omega} |\nabla u_{\varepsilon}^{2}|^{\beta}
&= 2^{\beta} \int_{0}^{T} \int_{\Omega}
\left(u_{\varepsilon}^{\gamma-1} v_{\varepsilon} |\nabla u_{\varepsilon}|^{2}\right)^{\frac{\beta}{2}}
u_{\varepsilon}^{\frac{(3-\gamma)\beta}{2}} v_{\varepsilon}^{-\frac{\beta}{2}} \\
&\le 2^{\beta} \int_{0}^{T} \int_{\Omega}
u_{\varepsilon}^{\gamma-1} v_{\varepsilon} |\nabla u_{\varepsilon}|^{2}
\;+\;
2^{\beta} \int_{0}^{T} \int_{\Omega}
u_{\varepsilon}^{\frac{(3-\gamma)\beta}{2-\beta}}
v_{\varepsilon}^{-\frac{\beta}{2-\beta}},
\qquad \text{for all } T>0,\ \varepsilon \in (0,1).
\end{align*}
Invoking Lemmas~\ref{lemma-4.4}, \ref{lemma-4.6} and \eqref{615-1428}, we conclude that
\begin{align}\label{-4.23}
\left(u_{\varepsilon}^{2}\right)_{\varepsilon \in (0,1)}
\ \text{is bounded in }\
L^{\beta}\!\left((0,T);\, W^{1,\beta}(\Omega)\right)
\qquad \text{for all } T>0 .
\end{align}
The regularity conditions \eqref{-2.1} and \eqref{-2.2} in Definition~\ref{def-weak-sol} follow directly from \eqref{-4.15}, \eqref{-4.16}, and \eqref{-4.23}.  
Similarly, combining \eqref{-4.15}, \eqref{-4.16}, \eqref{-4.17}, and \eqref{-4.23}, we deduce \eqref{-2.3} and \eqref{-2.4}.
\end{proof}

\begin{proof}[Proof of Theorem \ref{thm-1.1}]
Theorem~\ref{thm-1.1} follows directly from Lemmas~\ref{lemma-4.1}, \ref{lemma-4.4}, \ref{lemma-4.5}, and \ref{lemma-4.9}.
\end{proof}

\noindent\textbf{Data availability} The manuscript has no associated data.

\vskip 3mm
\noindent {\large\textbf{Declarations}}
\vskip 2mm 

\noindent\textbf{Conflict of interest} On behalf of all authors, the corresponding author states that there is no conflict of interest.

\hfill$ \Box$


\end{document}